\documentclass[13pt, notitlepage]{article}
\usepackage{amssymb,amsmath,comment}
\usepackage{mathtools}
\catcode`\@=11 \@addtoreset{equation}{section}
\def\thesection{\arabic{section}}

\def\theequation{\thesection.\arabic{equation}}
\catcode`\@=12
\usepackage{enumitem}
 \usepackage[normalem]{ulem}
\usepackage{graphics, color, tikz, booktabs}
\usepackage{colortbl}%
\usepackage{a4wide}
\usepackage{parskip}
\usepackage{hyperref}

\newcommand{\Om} {\Omega}

\newcommand{\De} {\Delta}

\newcommand{\noi} {\noindent}
\newcommand{\na} {\nabla}

\newcommand{\mb} {\mathbb}
\newcommand{\mc} {\mathcal}

\newcommand{\X}{\mathcal{X}_0^{1,2}}

\newcommand{\ld} {\langle}
\newcommand{\rd} {\rangle}
\newcommand{\I}{\int\limits_}

\usepackage[all]{xy}
\catcode`\@=11
\newcommand{\om} {\Omega}

\def\theequation{\@arabic{\c@section}.\@arabic{\c@equation}}
\catcode`\@=12

\def\QED{\hfill {$\square$}\goodbreak \medskip}
\newtheorem{Theorem}{Theorem}[section]
\newtheorem{Lemma}[Theorem]{Lemma}
\newtheorem{Proposition}[Theorem]{Proposition}

\newtheorem{Remark}[Theorem]{Remark}
\newtheorem{Definition}[Theorem]{Definition}

\def\XXint#1#2#3{{\setbox0=\hbox{$#1{#2#3}{\int}$ }
		\vcenter{\hbox{$#2#3$ }}\kern-.6\wd0}}

\begin{document}
	{\vspace{0.01in}
		\title{On Coron problems with Choquard term and mixed operator}
		\author{
Jacques Giacomoni$^{1}$, 
Tuhina Mukherjee$^{2}$\thanks{Corresponding author: \texttt{tuhina@iitj.ac.in}}, 
Lovelesh Sharma$^{3}$ \\
\small $^{1}$Laboratoire de Mathématiques et de leurs Applications, Université de Pau et des Pays de l’Adour, France \\
\small $^{2}$Department of Mathematics, Indian Institute of Technology Jodhpur, Rajasthan 342030, India.\\
\small $^{3}$Department of Mathematics, Indian Institute of Technology Delhi, Hauz Khas, Delhi 100016, India.
}
		\date{}
		
		\maketitle
		
	\begin{abstract}
In this article, we study a Coron-type problem involving a critical Choquard nonlinearity driven by a mixed operator combining the Laplacian and fractional Laplacian. In annular-type domains, we prove the existence of nontrivial positive solutions when the inner hole is sufficiently small. Using variational methods and concentration compactness arguments, we establish a global compactness result for Palais- Smale sequences and obtain high-energy solutions using topological methods. 
We also derive regularity results for weak solutions. 
\vspace{0.5cm}

\noi \textbf{\textit{2020 Mathematics Subject Classification:}} 35A01, 35A15, 35B33, 35B65.

\vspace{0.2cm}

\noi \textbf{\textit{Keywords:}} Mixed operator, Choquard nonlinearity, Global Compactness,   Profile decomposition, Coron problem, High-energy solutions, Regularity result.
\end{abstract}	
\section{Introduction and main results}\label{I}

The aim of the present article is to study the following problem consisting of a combination of local and nonlocal operators along with critical Choquard nonlinearity 
\begin{equation}\label{1}
\begin{cases}
\displaystyle \mathcal{L}u(x)
=
\left(\int_{\Omega}\frac{u(y)^{2_\mu^\ast}}{|x-y|^\mu}\,dy\right)
u(x)^{2_\mu^\ast-1}
&\text{in }\Omega,\\[6pt]
~~\quad~u>0\ \text{in }\Omega,\\
~~~~\quad u=0\ \text{in }\mathbb{R}^n\setminus\Omega,
\end{cases}\tag{$P$}
    \end{equation}
% \begin{align}
%     \begin{cases}	
%         \mathcal{L} u = \left(\displaystyle\int_{\Omega}\frac{|u(y)|^{2_\mu^\ast}}{|x-y|^\mu}dy\right)u(x)^{2_\mu^\ast-1} & \text{in } \Omega, \\
%         ~~u>0 & \text{in } \Omega, \\
%         ~~u = 0 & \text{in } \mathbb{R}^n \setminus \Omega,
%     \end{cases} \tag{P}
% \end{align}
where $\Om$ is a bounded domain of $\mb R^n$ with $C^{1,1}$ boundary $\partial \Om$, $n \geq 3$, $0 <\mu <n$, $2_\mu^\ast =(2n-\mu)/(n-2)$. The mixed operator $\mc L$ in \eqref{1} is given by 
\begin{equation}\label{ecl}
	\mc L =-\Delta +(-\Delta)^s~\text{for some}~ s \in (0,1).
\end{equation}
The word "mixed," which refers to the type of operator combining both local and nonlocal features and to the differential order of the operator. The operator $\mc  L$ is obtained by the superposition of the classical Laplacian $(-\Delta)$ and the fractional Laplacian $(-\Delta)^s$, which for a fixed parameter $s \in (0,1)$, is defined by 
\begin{align*}
	(-\Delta)^su = C(n,s)P.V. \int_{\mb R^n} \frac{u(x)-u(y)}{|x-y|^{n+2s}}dy.
\end{align*}
The term $"P.V"$ stands for Cauchy's principal value and $C(n,s)$ is a normalizing constant whose explicit expression is given by
\begin{equation*}
	C(n,s)=\left(\int_{\mb R^n}\frac{1-\cos(z_1)}{|z|^{n+2s}}dz\right)^{-1}.
\end{equation*}
The investigation of mixed operators of the form $\mc L$ in \eqref{ecl} is driven by their broad applicability in various fields. Such operators naturally emerge in the applied sciences to describe the combined effects of local and nonlocal variations in physical phenomena. They serve as effective models for diffusion processes occurring at different time scales: the higher-order operator typically dominates the diffusion at small times, while the lower-order operator becomes significant over larger time scales. Notably, they appear in bi-modal power law distribution processes \cite{PV1}. Additional applications can be found in optimal search theory, biomathematics, and animal foraging models \cite{DLV}, as well as in several other contexts \cite{KLS, MV} and the references therein.

Motivated by both their wide range of applications and intrinsic mathematical interest, elliptic problems involving mixed operators with local and nonlocal features have recently attracted significant attention. Current research has primarily focused on questions related to the existence and regularity of solutions. In the following, we provide a brief overview of the literature concerning problems of the type
\[
-\Delta_r u + (-\Delta)_r^s u = g \quad \text{in } \Omega,
\]
where $\Omega \subset \mathbb{R}^n$ is a domain, $r \in (1, \infty)$, $s \in (0,1)$, and $-\Delta_r$ and $(-\Delta)_r^s$ denote the $r$-Laplacian and fractional $r$-Laplacian operators, respectively.

In the linear case $r = 2$, the structural properties of solutions, such as the existence of weak solutions, the strong maximum principle, local boundedness, interior Sobolev and Lipschitz regularity, and other qualitative features, have been studied in \cite{biagi2022mixed}. Other aspects, including the symmetry and properties of the first eigenvalue, are discussed in \cite{ DLV}.
For the nonlinear case $r \neq 2$ with $g = 0$, Garain and Kinnunen \cite{garain2022regularity} established regularity results for weak solutions, including local boundedness, Harnack inequalities, local Hölder continuity, and semicontinuity. In the inhomogeneous setting, boundedness and the strong maximum principle have been studied in \cite{biagi2024brezis} (see also \cite{biagi2021hong}). 

For the classical Laplacian operator $\mathcal{L} = -\Delta$ in \eqref{1}, the Coron problem was investigated by Divya et al. \cite{goel2020coron}, where the authors established the existence of a positive high-energy solution in an annular-type domain. Later, Secchi et al. \cite{secchi_coronpb_fractional} studied the Coron problem in the framework of the fractional Laplacian.
This work is motivated by the study of nonlinear problems with critical exponents, typically modelled as
\begin{equation}\label{model}
\mathcal{L}u = X(u) + \lambda Y(u) \quad \text{in } \Omega,
\end{equation}
where $\Omega \subset \mathbb{R}^n$ is a domain, $\mathcal{L}$ is a (local/nonlocal/mixed) operator, ${X}$ is a critical nonlinearity, ${Y}$ is a subcritical perturbation, and $\lambda \in \mathbb{R}$ is a parameter. A key example in the local case is the Brezis-Nirenberg problem:
\[
-\Delta u = |u|^{2^* - 2}u + \lambda u^q, \quad u|_{\partial \Omega} = 0,
\]
where $2^* = \frac{2n}{n-2}$ is the critical Sobolev exponent. For $\lambda = 0$ and star-shaped $\Omega$, no nontrivial solutions exist. However, for suitable $\lambda$, existence and multiplicity results have been established via careful minimax estimates, see \cite{brezis1983positive,ambrosetti1994combined}. 

In the nonlocal setting where $\mathcal{L} = (-\Delta)^s$ for $s \in (0,1)$, critical problems have been extensively studied. Servadei and Valdinoci \cite{servadei2015brezis} pioneered this direction for linearly perturbed problems with the critical exponent $2^*_s = \frac{2n}{n-2s}$. Barrios et al. \cite{barrios2015critical} later extended these results to cover both sublinear and superlinear perturbations. The case of critical Choquard nonlinearity $X(u)$ was treated by Mukherjee and Sreenadh \cite{mukherjee2016fractional}. They established existence results contingent on the dimension $n$ and the parameter $\lambda$, particularly distinguishing between the cases $n \geq 4s$ and $2s < n < 4s$.

In contrast, critical Choquard problems for mixed operators $\mathcal{L}$ remain relatively unexplored in the annular type domain. The foundational work of Biagi et al. \cite{biagi2025brezis} established a mixed Sobolev inequality, proving that the associated best Sobolev constant coincides with the classical constant. Recently, this analysis was extended to the Choquard nonlinearity in \cite{anthal2023choquard}, where the authors demonstrated that the corresponding critical constant for the mixed Choquard problem also equals the classical Choquard constant, leading to analogous existence results. This foundational result was then leveraged to prove the existence of solutions for both linear and superlinear subcritical perturbations.
More recently, Chakraborty et al. \cite{chakraborty2025global} obtained a global compactness result for the Brezis–Nirenberg type problem involving a mixed local–nonlocal operator. Motivated by these results, a natural open question arises to investigate the Coron problem for elliptic equations involving mixed operators with critical Choquard-type nonlinearities and parameter $\lambda=0$ in \eqref{model}. This is precisely the main objective of the present paper. The case of the critical Choquard nonlinearity,
\[
X(u) = \left(\int_{\Om} \frac{|u(y)|^{2_\mu^\ast}}{|x-y|^\mu}\,dy\right)|u|^{2_\mu^\ast-2}u,
\]
together with linear, sublinear, or superlinear perturbations, has been studied by Gao and Yang in \cite{gao2016brezis} and \cite{GY}, where they established existence and multiplicity results depending on certain ranges of $\lambda$. The critical exponent $2_\mu^\ast$ arises naturally from the well-known Hardy-Littlewood-Sobolev inequality, which plays a crucial role in applying variational methods to problem \eqref{1}.

\begin{Proposition}\cite{lieb2001analysis}\label{HLS}
Let $t, r > 1$ and $0 < \mu < n$ with $\frac{1}{t} + \frac{\mu}{n} + \frac{1}{r} = 2$, $g \in L^{t}(\mathbb{R}^n)$ and $h \in L^{r}(\mathbb{R}^n)$. \textit{There exists a sharp constant } $C(t,n,\mu,r)$, \textit{independent of } $g,h$ \textit{ such that}
\begin{equation}\label{hls}
\int_{\mathbb{R}^n} \int_{\mathbb{R}^n} \frac{g(x)h(y)}{|x-y|^{\mu}} \, dxdy \leq C(t,n,\mu,r) \|g\|_{L^t(\mathbb{R}^n)} \|h\|_{L^r(\mathbb{R}^n)}.
\end{equation}
\textit{If } $t = r = \frac{2n}{2n-\mu}$, \textit{ then } $C(t,n,\mu,r) = C(n,\mu) = \pi^{\mu/2} \frac{\Gamma\left(\frac{n-\mu}{2}\right)}{\Gamma\left(n-\frac{\mu}{2}\right)} \left\{ \frac{\Gamma\left(\frac{n}{2}\right)}{\Gamma(n)} \right\}^{-1+\mu/n}$. \textit{In this case there is equality in \eqref{hls} if and only if } $g \equiv \text{(constant)}\cdot h$ \textit{ and}
\[
h(x) = A(\gamma^{2} + |x-a|^{2})^{-\frac{(2n-\mu)}{2}}
\]
\textit{for some } $A \in \mathbb{C}$, $0 \neq \gamma \in \mathbb{R}$ \textit{ and } $a \in \mathbb{R}^n$.
\end{Proposition}

In general, letting $g = h = |u|^{q}$ in the Hardy-Littlewood-Sobolev inequality, we get
\[\int_{\mathbb{R}^n} \int_{\mathbb{R}^n} \frac{|u(x)|^{q} |u(y)|^{q}}{|x-y|^{\mu}} dxdy\]
is well defined if $|u|^{q} \in L^{t}(\mathbb{R}^n)$ for some $t > 1$ satisfying
\[\frac{2}{t} + \frac{\mu}{n} = 2.\]
Thus, for $u \in H^{1}(\mathbb{R}^n)$, by Sobolev Embedding theorems, we must have
\[\frac{2n-\mu}{n} \leq q \leq \frac{2n-\mu}{n-2}.\]
Thus, $\frac{2n-\mu}{n}$ is called the lower critical exponent and $2^*_{\mu}=\frac{2n-\mu}{n-2}$ is the upper critical exponent in the sense of the Hardy-Littlewood-Sobolev inequality. From this, for $u \in H^1(\mb R^n) $ we have
\[
\left(\int\limits_{\mb R^n}\int\limits_{\mb R^n}\frac{|u(x)|^{2_\mu^\ast}|u(y)|^{2_\mu^\ast}}{|x-y|^\mu}dxdy \right)^\frac{1}{2_\mu^\ast} \leq C(n,\mu)^\frac{1}{2_\mu^\ast}  \|u\|^2_{L^{2^*}(\mathbb{R}^n)}.
\]

Current research has focused particularly on the existence of solutions in annular-type domains when the inner hole is sufficiently small and has also established important regularity properties, showing that solutions belong to $L^{\infty}(\Omega)$ and, moreover, to $C^{2,\beta}(\overline{\Omega})$ every \(\beta\) such that 
\(0 < \beta < \min\{s,\alpha,n-\mu\}\) (see Theorem \ref{thm:7.1} and Theorem \ref{thm1.5}). Drawing on foundational results in profile decomposition namely, Struwe's analysis for the Laplacian \cite{Struwe-Z} and the work of Palatucci et al. for the fractional Laplacian \cite{Pala-Pis-Non}, we investigate the profile decomposition of Palais-Smale sequences corresponding to an energy functional that involves both the Laplacian and fractional Laplacian operators. Since Struwe's seminal contribution \cite{Struwe-Z}, the well-known Global Compactness theorem in the Sobolev space $H^1$ has become a fundamental tool in analysis, supporting existence results in various areas such as ground state solutions of nonlinear Schrödinger equations, Yamabe-type problems in conformal geometry, and the prescription of $Q$-curvature. In Theorem~\ref{PS_decomposition}, we extend Struwe's Global Compactness framework to our setting by providing a profile decomposition for Palais–Smale sequences associated with the energy functional $I$ defined in \eqref{eqI}. The main difficulty in our analysis arises from the lack of a common noncompact scaling group that acts on the space $\X(\Omega)$, which, unlike the classical Sobolev space, makes the decomposition process more complicated. In this article, we establish a global compactness result by introducing the notion of Morrey spaces; see definition \eqref{Ms}.
Finally, using the concentration–compactness principle together with the deformation lemma, we prove the existence of a positive high-energy solution. To the best of our knowledge, there is no previous work on the Coron problem for Choquard equations involving mixed operators. Motivated by this discussion, in the present work, we study problem \eqref{1}. Our goal in this paper is to fill this gap.

A central idea in their proof was the use of rescaled functions $u_t(x) = t^{\frac{n-2}{2}} u(tx)$, which allowed the nonlocal term to be effectively diminished. In particular, the Gagliardo seminorm satisfies
\[
[u_t]_s^2 = t^{2s-2}[u]_s^2 \to 0 \quad \text{as } t \to \infty,
\]
combined with the compact embedding of $\X(\Omega)$ into $L^2(\Omega)$, this observation was crucial in identifying the limiting profiles in Struwe's global compactness result, see \cite{struwe2000variational, goel2020coron}, where the profiles correspond to solutions of problem \eqref{eq4.7}.  
This naturally led to the question of whether the space $\X(\Omega)$ admits a compact embedding into $H^s(\Omega)$, which was answered affirmatively in [\cite{chakraborty2025global}, Lemma \ref{compact_embedding}]. Consequently, following the approach in \cite{chakraborty2025global} and \cite{goel2020coron}, this allowed us to reduce to the Laplacian case and establish the following global compactness result.

 \begin{Theorem}\label{PS_decomposition}  
Let $\{u_k\}_{k \in \mathbb{N}} \subset \mathcal{X}^{1,2}_0(\Omega) $ be such that $I(u_k) \to c$, $I'(u_k) \to 0$. Then, passing, if necessary, to a subsequence, there exists a nonnegative solution $v_0 \in \mathcal{X}^{1,2}_0(\Omega)  $ of
\begin{equation}\label{eq4.6}
    \mathcal{L} u = \left(\int_{\Omega} \frac{u(y)^{2^*_{\mu}}}{|x-y|^{\mu}} \, dy \right) u^{2^*_{\mu}-1}  \quad \text{in } \Omega 
\end{equation}
and  $k \in \mathbb{N} \cup \{0\}$, nonnegative nontrivial solutions $v_1, v_2, \ldots, v_k$ of
\begin{equation}\label{eq4.7}
    (-\Delta) u = (|x|^{-\mu} * u^{2^*_{\mu}}) u^{2^*_{\mu}-1}  \quad \text{in } \mathbb{R}^n 
\end{equation}
with $v_i \in D^{1,2}(\mathbb{R}^n)$ and $k$ sequences $\{y_k^i\}_{k \in \mathbb{N}} \subset \mathbb{R}^n$ and $\{\lambda_k^i\}_{k \in \mathbb{N}} \subset \mathbb{R}_+$, $i=1,2,\ldots,k$,
satisfying
\[
\frac{1}{\lambda_k^i} \operatorname{dist}(y_k^i, \partial \Omega) \to \infty
\quad \text{and} \quad 
    \left\|u_k - v_0 - \sum_{i=1}^k (\lambda_k^i)^{\frac{2-n}{2}} v_i\left( \frac{x - y_k^i}{\lambda_k^i} \right)\right\|_{D^{1,2}(\mathbb{R}^n)} \to 0 \quad \text{as } k \to \infty,
\]
and
\begin{equation}\label{eq4.8}
 I(u_k) \to  I(v_0) + \sum_{i=1}^k J(v_i) = c~~~\text{as}~~k\to\infty 
    \end{equation}
    where energy functionals $I$ and $J$ are defined in \eqref{eqI} and \eqref{Ener-Inf} respectively.
\end{Theorem}

As shown in Theorem \ref{thm:best_constant}, the best constant $S_{H,L,M}$ in the associated Sobolev inequality is never attained, which implies that ground state solutions to the critical Choquard problem do not exist. In such circumstances, profile decomposition becomes an essential tool, as it precisely identifies the levels at which compactness breaks down. This framework not only clarifies the underlying concentration phenomena but also determines the energy thresholds where one must seek higher energy solutions.
Guided by this observation, the search for high-energy solutions arises naturally. Two classical strategies are commonly employed to address the lack of compactness: one introduces suitable perturbations into the equation, while the other exploits the topology or geometry of the domain. The latter gives rise to the celebrated Coron problem, which demonstrates how topological features of the domain can restore compactness and thereby guarantee the existence of solutions. A substantial body of work has developed along these lines, employing profile decomposition to establish existence and multiplicity results; see, for example, \cite{  sergio_multiple_soln}.

The study of critical elliptic equations on non-contractible domains was initiated by Coron in 1984. In particular, Coron \cite{coron1984topologie} proved the existence of a positive solution to the critical problem  
\begin{equation} \label{bh83}
\left\{
\begin{aligned}
-\De u &= u^{\frac{n+2}{n-2}}, \quad u > 0 \quad \text{in } \Om, \\
u &= 0 \qquad\qquad\qquad~\text{on } \partial \Om,
\end{aligned}
\right.
\end{equation}
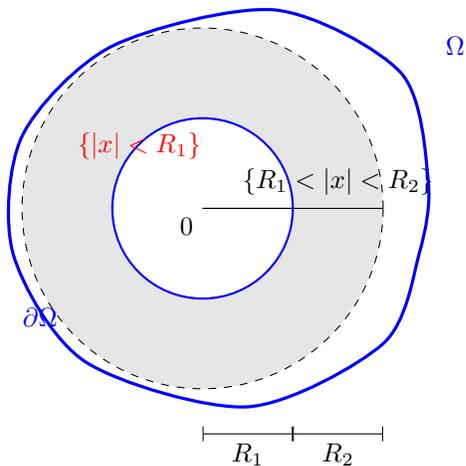
\begin{figure}[h!]
    \centering
    \begin{tikzpicture}[scale=1.2]
        % Parameters (for clarity)
        % R1 = 1cm, R2 = 2cm in this picture

        % Draw the annulus (the region that MUST be inside Omega)
        \filldraw[fill=gray!20, draw=black, dashed] (0,0) circle (2cm);
        \filldraw[fill=white, draw=black] (0,0) circle (1cm);
        \node at (1.5, 0.3) {$\{R_1 < |x| < R_2\}$};

        % Draw the smooth domain Omega (a wobbly curve that obeys the conditions)
        \draw[blue, very thick, smooth cycle] plot coordinates {
            (2.5, 0.3) (2.2, 1.5) (0.8, 2.2) (-1.0, 1.8) (-2.0, 0.8)
            (-2.1, -0.5) (-1.3, -1.7) (0.5, -2.2) (2.0, -1.5) (2.4, -0.5)
        };
        \node[blue] at (2.8, 1.8) {$\Omega$};
        \node[blue] at (-1.8, -1.2) {$\partial\Omega$};

        % Draw the inner ball (red label, blue boundary)
        \draw[blue, thick] (0,0) circle (1cm);
        \node[red] at (-0.7, 0.7) {$\{ |x| < R_1 \}$};

        % Draw the origin and labels for R1 and R2
        \draw[black] (0,0) node[anchor=north east] {$0$} -- (1,0);
        \draw[black] (0,0) -- (2,0);
        \draw[|-|] (0, -2.5) -- node[below] {$R_1$} (1, -2.5);
        \draw[|-|] (1, -2.5) -- node[below] {$R_2$} (2, -2.5);
    \end{tikzpicture}
    \caption{Geometric condition on the domain $\Omega$. The annulus $\{R_1 < |x| < R_2\}$ (gray region) is contained within $\Omega$ (blue boundary). The ball $\{ |x| < R_1 \}$ (red label, blue boundary) is \textbf{not} entirely contained in $\overline{\Omega}$.}
    \label{fig:domain_condition}
\end{figure}

Later, Bahri and Coron \cite{bahri1987nonlinear} extended this result by showing that if there exists a positive integer $d$ with $H_d(\Om, \mb Z_2) \neq 0$, then problem \eqref{bh83} also admits a positive solution.  
Since then, several extensions of the Coron problem have been studied for different classes of nonlinear operators. For instance, the $p$ Laplacian case was investigated in \cite{mercuri_coron_p_lap}, the fractional Laplacian in \cite{secchi_coronpb_fractional}, and analogous results for systems such as the Lane–Emden system were obtained in \cite{jin_coron_lane_emden}.  
In view of the profile decomposition established in Theorem \ref{PS_decomposition} for the Choquard critical problem \eqref{1}, we now focus on proving a Coron-type result for our setting. By developing an appropriate variational framework and applying concentration–compactness methods, we derive the following existence theorem for high-energy solutions.  

\begin{Theorem}\label{mainthm}
Assume that $\Om \subset \mb R^n$ is a bounded domain satisfying the conditions  
\begin{align}
(i)& \quad \{x \in \mb R^n : R_1 < |x| < R_2 \} \subset \Om,\label{eq1.6}\\
(ii)& \quad \{x \in \mb R^n : |x| < R_1 \} \not\subset \overline{\Om},\label{eq1.7}
\end{align}
for some $0 < R_1 < R_2 < +\infty$. Then there exists $n_0 \in \mb N$ such that for all $n \geq n_0$, provided the ratio $\tfrac{R_2}{R_1}$ is sufficiently large (depending on $n$), the problem \eqref{1} admits a positive high-energy solution in $\mc X^{1,2}_0(\Om)$ satisfying 
\[
\frac{1}{2}\left(\frac{n-\mu+2}{2n-\mu}\right) S^{\frac{2n-\mu}{n-\mu+2}}_{H,L,C} 
< I(u) < 
\left(\frac{n-\mu+2}{2n-\mu}\right) S^{\frac{2n-\mu}{n-\mu+2}}_{H,L,C}.
\]
\end{Theorem}
%%%%%%%%%%%%%%%%%%%%%%%%%%%%%%%%%%%%%%%%%%%%%%%%%%%%%%%%%
The proof of Theorem \ref{mainthm} is carried out by contradiction. To this end, we introduce a two-parameter family of test functions depending on $R \gg 1$, define as  
\begin{equation}\label{familyOfBubbles:Intro}
\mc A_R = \left\{ \phi_R \mathcal{U}[t\sigma,(1-t)] : \phi_R \in C^{\infty}_c(\mb R^{n}), \; t \in [0,1), \; \sigma \in \Sigma \right\},
\end{equation}
where \(\Sigma := \{ x \in \mb R^n : |x| = 1 \} \), $\mathcal{U}$ denotes the standard Aubin–Talenti bubble (see \eqref{eq5.1}), $\phi_R$ is a cut-off function, and parameters $t$ and $\sigma$ encode the concentration and translation of the bubbles within the domain. We then estimate the energy associated with this family and show that, for sufficiently large dimension $n$ and an appropriately large ratio $\tfrac{R_2}{R_1}$, it remains strictly below the threshold  
\(
2^{\frac{n-\mu+2}{2n-\mu}} S_{H,L,C} \quad \text{(see Lemma \ref{lem4.6}).}
\)  
This estimate is crucial in proving the existence of a high-energy solution, as it permits a deform the bubble family whose energy approaches the first critical level. Assuming, by contradiction, that no positive nontrivial solution exists, this level $2^{\frac{n-\mu+2}{2n-\mu}} S_{H,L,C}$ becomes the next critical threshold after $S_{H,L,C}$ at which the Palais–Smale condition fails. Basically, in the variational analysis of problem \eqref{1}, we work with two closely related functionals: the energy functional $I(u)$ and the Sobolev-type quotient functional $S(u, \Omega)$ defined in \eqref{eqS}. While the critical points of $I$ correspond directly to solutions of \eqref{1}, the scale-invariant nature of $S(\cdot, \Omega)$ is more suitable for topological arguments. A crucial bridge between these two frameworks is established in Lemma \ref{lem2.88}, which shows that a Palais-Smale sequence for $S(\cdot, \Omega)$ at level $c$ corresponds to a Palais-Smale sequence for $I$ at level $\frac{n-\mu+2}{2(2n-\mu)} c^{\frac{2n-\mu}{n-\mu+2}}$. This relationship maps the non-compactness threshold from $\left(S_{H,L,C}, 2^{\frac{n-\mu+2}{2n-\mu}} S_{H,L,C}\right)$ for $S(\cdot, \Omega)$ to $(\beta, 2\beta)$ for $I$, where $\beta = \frac{1}{2} \left( \frac{n-\mu+2}{2n-\mu} \right) S_{H,L,C}^{\frac{2n-\mu}{n-\mu+2}}$ is the least energy, (see Lemma \ref{1lem2.8}).

Finally, by applying a deformation lemma, we construct a homotopy that retracts the unit sphere $\Sigma \subset \Om$ continuously to a single point in $\Om$. This, however, contradicts the topological condition \eqref{eq1.7}.

We state our regularity results below.
\begin{Theorem}\label{thm:7.1}
Let \(0 \leq u \in \mathcal{X}^{1,2}_0(\Omega)\) be a weak solution of \eqref{1}, that is,
\[
\int_{\mathbb{R}^n} \nabla u \cdot \nabla \varphi \, dx 
+ \int_{\mathbb{R}^{2n}} \frac{(u(x) - u(y))(\varphi(x) - \varphi(y))}{|x - y|^{n + 2s}} \, dxdy
= \int_{\Omega} \int_{\Omega} 
\frac{|u_+(y)|^{2^*_\mu} |u_+(x)|^{2^*_\mu - 1}  \varphi(x)}{|x - y|^{\mu}} \, dydx,
\]
for every test function \(\varphi \in {\mathcal{X}^{1,2}_0(\Omega)}\).  
Then \(u \in L^\infty(\mathbb{R}^n)\).
\end{Theorem}

\begin{Theorem}\label{thm1.5}
Let $u \in \mathcal{X}^{1,2}_0(\Omega)$ be a weak solution of \eqref{1}. 
Then $u \in W^{2,r}(\Omega)$ for every $r \in (1,\infty)$ if $s \in (0,1/2]$, 
and for every $r \in \left(0,\frac{n}{2s-1}\right)$ if $s \in (1/2,1)$. 
Consequently, $u \in C^{1,\beta}(\overline{\Omega})$ for all $\beta \in (0,1)$ when $s \in (0,1/2]$, 
and for all $\beta \in (0,2-2s)$ when $s \in (1/2,1)$. 
Moreover, if $\Omega$ is of class $C^{2,\alpha}$ with $\alpha \in (0,1)$, 
then $u \in C^{2,\beta}(\overline{\Omega})$ for every $\beta$ satisfying 
$0 < \beta < \min\{s,\alpha,n-\mu\}$.
\end{Theorem}

\textbf{Plan of the paper:} The rest of the paper is organized as follows. In Section 2, we introduce the functional framework, notation, and preliminary results essential for our analysis. In particular, we establish the weak maximum principle. In Section 3, we established certain regularity results for weak solutions of problem~\eqref{1}, 
showing in particular that \(u \in L^{\infty}(\Omega)\) and 
\(u \in C^{2,\beta}(\overline{\Omega})\) for every \(\beta\) such that 
\(0 < \beta < \min\{s,\alpha,n-\mu\}\). Also, we proved a strong maximum principle(Proposition \eqref{sm}). In Section 4, we establishes the Palais-Smale decomposition and the global compactness result for Palais-Smale sequences related to the mixed local-nonlocal operator with Choquard nonlinearity. Furthermore, in Section 5, 
 we prove existence results for positive high-energy solutions in annular-type domains using variational methods and topological arguments. 
 
 \textbf{Notations.} Throughout the paper, we shall use the following notations.  
\begin{itemize}
\item Positive constants will be denoted by $C, C_1, C_2, \dots$. Their values may vary from line to line, and also denote some constants by $D.$  
    % \item The standard norm on $L^p(\mb R^{n})$ is denoted by $|\cdot|_p$.  
    \item For $x \in \mb R^{n}$ and $r > 0$, we denote by $B_r(x)$ the ball of radius $r$ centered at $x$.  In particular, $B_R := B(0,R)$. 
    \item For any function $u$, we denote its positive and negative parts by 
    \(
        u_+ := \max\{u, 0\}, \quad u_- := \max\{-u, 0\}.
    \)
    \item Let $X= \X(\Omega)$, $X'$ denotes the topological dual of the fractional  Sobolev space $X$. It is endowed with the dual operator norm induced by $\X(\Om)$. 
    \item Throughout the paper, \(\langle \cdot , \cdot \rangle\) denotes the duality pairing between 
\(X'\) and \(X\).
\end{itemize}

 \section{Preliminaries}\label{P}

In this section, we introduce the functional framework needed to study problem \eqref{1}, together with some notations and preliminary results that will be used throughout the paper. For a more detailed exposition of these settings, we refer to \cite{biagi2025brezis} and the references therein.  

Let $s \in (0,1)$. For a measurable function $u:\mathbb{R}^{n} \to \mathbb{R}$, we define the \emph{Gagliardo seminorm} of order $s$ by  
\begin{equation*}
    [u]_s := \left( \int_{\mathbb{R}^{2n}} \frac{|u(x)-u(y)|^2}{|x-y|^{n+2s}} \, dx \, dy \right)^{\!\frac12}.
\end{equation*}

Let $\Omega \subset \mathbb{R}^{n}$ be a non-empty open set (not necessarily bounded). We define the space $\mathcal{X}^{1,2}_0(\Omega)$ as the completion of $C_0^\infty(\Omega)$ with respect to the following norm  
\begin{align*}
    \eta(u) := \left( \|\nabla u\|^2_{L^2(\mathbb{R}^{n})} + [u]_s^2 \right)^{\frac12}, 
    \qquad u \in C_0^\infty(\Omega),
\end{align*}
and 
$$\langle u, \varphi \rangle_{\eta} 
= \int_{\mathbb{R}^n} \nabla u \cdot \nabla \varphi \, \mathrm{d}x 
+ \int_{\mathbb{R}^{2n}} \frac{(u(x) - u(y))(\varphi(x) - \varphi(y))}{|x-y|^{n+2s}} \, \mathrm{d}x \, \mathrm{d}y~~\forall~~\varphi\in C_0^\infty(\Omega).$$
%	\begin{Remark}\label{remarkdef}
%Although $u \in C^{\infty}_0(\Omega)$ (meaning it vanishes identically outside $\Omega$), the definition of the norm $\eta(\cdot)$ incorporates the $L^2$-norm of $\nabla u$ over the entire space $\mathbb{R}^n$. This emphasizes that the functions in $\mathcal{X}^{1,2}_0(\Omega)$ are naturally considered as elements defined on the whole space, not only on $\Omega$.  
%The advantage of adopting such a global functional setting is that smooth compactly supported functions in $\Omega$ can approximate elements of $\mathcal{X}^{1,2}_0(\Omega)$ in $\mathbb{R}^n$ with respect to $\eta(\cdot)$.  

%In particular, when $\Omega \neq \mathbb{R}^n$, this global characterization ensures that functions in $\mathcal{X}^{1,2}_0(\Omega)$ automatically satisfy the boundary condition of problem \eqref{1}, namely
%\begin{equation}\label{e2.2}
%u \equiv 0 \quad \text{a.e. in } \mathbb{R}^n \setminus \Omega, \quad \forall~ u \in %\mathcal{X}^{1,2}_0(\Omega).
%\end{equation}

%To justify \eqref{e2.2}, we distinguish the following two cases.
Below are two characterizations of the set $\mathcal{X}^{1,2}_0(\Omega)$ based on features of $\Omega$.
\begin{enumerate}
    \item If $\Omega$ is bounded, then the norm $\eta(\cdot)$ and the $H^1$ norm in $\mathbb{R}^n$ are equivalent on $C^{\infty}_0(\Omega)$. Consequently,
    \begin{align*}
        \mathcal{X}^{1,2}_0(\Omega)  
        &= \overline{C^{\infty}_0(\Omega)}^{\| \cdot \|_{H^1(\mathbb{R}^n)}} \\
        &= \big\{ u \in H^1(\mathbb{R}^n) : \; u|_{\Omega} \in H^1_0(\Omega), \; u \equiv 0 \ \text{a.e. in } \mathbb{R}^n \setminus \Omega \big\}.
    \end{align*}
    
    \item If $\Omega$ is unbounded, then
    \[
    \mathcal{X}^{1,2}_0(\Omega) 
    = \big\{ u \in L^{2^*}(\mathbb{R}^n) : \; u \equiv 0 \ \text{in } \mathbb{R}^n \setminus \Omega,\; \nabla u \in L^2(\mathbb{R}^n),\; [u]_s^2 < \infty \big\}.
    \]
In particular, for $\Omega = \mathbb{R}^n$, one obtains
    \[
    \mathcal{X}^{1,2}_0(\mathbb{R}^n) 
    = \big\{ u \in L^{2^*}(\mathbb{R}^n) : \; \nabla u \in L^2(\mathbb{R}^n),\; [u]_s^2 < \infty \big\}.
    \]
\end{enumerate}
We also recall the following space
\[
D^{1,2}(\mathbb{R}^n)
= \overline{C_c^\infty(\mathbb{R}^n)}^{\;\|\nabla u\|_{L^2(\mathbb{R}^n)}}
= \left\{\, u \in L^{2^*}(\mathbb{R}^n) : \nabla u \in L^2(\mathbb{R}^n) \,\right\},
\]
where \( 2^* = \frac{2n}{n-2} \) is the critical Sobolev exponent and the associated norm is
\[
\|u\|_{D^{1,2}(\mathbb{R}^n)} = 
\left( \int_{\mathbb{R}^n} |\nabla u|^2 \, dx \right)^{1/2}.
\]
%\end{Remark}
In order to approach the original problem \eqref{1}, we begin by examining the following auxiliary problem 
\begin{equation}\label{2}
\begin{cases}
\displaystyle \mathcal{L}u(x)
= \left(\int_{\Omega} \frac{|u_{+}(y)|^{2_\mu^\ast}}{|x-y|^{\mu}}\,dy\right)
|u_{+}(x)|^{2_\mu^\ast-1} & \text{in }\Omega,\\[6pt]
% u>0 & \text{in }\Omega,\\[4pt]
~~\quad ~u=~~0 & \text{in }\mathbb{R}^{n}\setminus\Omega,
\end{cases}\tag{$P^*$}
\end{equation}}
where $u_+ = \max\{u,0\}$. We shall first study this problem \eqref{2}, which will later allow us to address the original problem~\eqref{1}, see Remark \ref{r2p}.

{ {In order to give a variational formulation to problem \eqref{1} and \eqref{2}, we introduce the notion of weak solutions.
\begin{Definition}
A function $u \in \mathcal{X}^{1,2}_0(\Omega)$ is called a weak solution of problem \eqref{1} if $u>0$ a.e. in $\Omega$ and for all $\varphi \in \mathcal{X}^{1,2}_0(\Omega)$, the following identity holds
\begin{align}
\int_{\mathbb{R}^{n}} \nabla u \cdot \nabla \varphi \, dx 
+ \int_{\mathbb{R}^{2n} } 
\frac{(u(x)-u(y))(\varphi(x)-\varphi(y))}{|x-y|^{n+2s}} \, dx dy
= \int_{\Omega}\int_{\Omega}
\frac{u(y)^{2_\mu^\ast}\,u(x)^{2_\mu^\ast-1}\varphi(x)}{|x-y|^\mu}\,dy\,dx.
\end{align}
\end{Definition}
\begin{Definition}\label{d4.1}
A function $u \in \mathcal{X}^{1,2}_0(\Omega)$ is called a weak solution of problem \eqref{2} if  for all $\varphi \in \mathcal{X}^{1,2}_0(\Omega)$, the following identity holds
\begin{align}\label{4.1}
\int_{\mathbb{R}^{n}} \nabla u \cdot \nabla \varphi \, dx 
+ \int_{\mathbb{R}^{2n} } 
\frac{(u(x)-u(y))(\varphi(x)-\varphi(y))}{|x-y|^{n+2s}} \, dx dy
= \int_{\Omega} \int_{\Omega}
\frac{|u_+(y)|^{2_\mu^\ast} |u_+(x)|^{2_\mu^\ast - 1}  \varphi(x)}{|x-y|^\mu} \, dx dy.
\end{align}
\end{Definition}}

 We also observe that Definition \ref{d4.1} is well-posed, in the sense that all the integrals in \eqref{4.1} are finite. Indeed, if $u, \varphi \in \mathcal{X}^{1,2}_0(\Omega)$, we have,
		\begin{align*}
		\left|\int_{ \mathbb{R}^{n}}  \na u\cdot \na \varphi \,dx + \int_{R^{2n}}\right. &\left.\frac{(u(x)-u(y))(\varphi(x)-\varphi(y))}{|x-y|^{n+2s}}dxdy\right|\\
			\leq& \|\nabla u\|_{L^2(\mathbb{R}^{n})}\|\nabla\varphi\|_{L^2(\mathbb{R}^{n})}+[u]_s[\varphi]_s \leq 2\mc \eta(u)\mc \eta(\varphi )<+\infty.
		\end{align*}
		Moreover, since $\mathcal{X}^{1,2}_0(\Omega) \hookrightarrow L^{2^*}(\mb R^n)$, using classical Hardy-Littlewood-Sobolev and H\"{o}lder's inequality(and taking into account that $u,v=0$ a.e. in $\mb R^n \setminus \Om$) we also have
		\begin{align*}
			\I{\Om}\I{\Om}\frac{|u_+(y)|^{2_\mu^*}|u_+(x)|^{2_\mu^*-1}\varphi(x)}{|x-y|^\mu}dxdy 
			\leq C(n,\mu)\|u\|^2_{L^{2^*}(\Omega)}\|\varphi\|^2_{L^{2^*}(\Omega)}<+\infty. 
		\end{align*}
\begin{Remark}\label{r2p}
 The right-hand side of \eqref{2} is nonnegative for every weak solution. 
Thus, by the weak maximum principle(Proposition \ref{weakm}), any weak solution $u$ of \eqref{2} satisfies $u\ge0$ a.e. in $\Omega$, 
and hence $u=u_{+}$.   
%Since $u = u_{+}$ for every weak solution, the nonlinearity in \eqref{2} simplifies to  
%\[
%|u_{+}(x)|^{2_\mu^\ast-1}=u(x)^{2_\mu^\ast-1},
%\qquad 
%\]
 Moreover, by applying the maximum principle(Proposition~\ref{sm}) and thanks to the regularity result {(Theorem~\ref{thm1.5})}, one can deduce that $u$ is a positive weak solution of \eqref{1}.

\end{Remark}

        The corresponding energy functional $I: \mathcal{X}^{1,2}_0(\Omega) \to \mathbb{R}$, associated to the problem \eqref{2} is given by
\begin{equation}\label{eqI}
I(u) = \frac{1}{2} \int_{\mathbb{R}^{n}} |\nabla u|^2\,dx+\frac{1}{2}\int_{\mathbb{R}^{2n}} \frac{|u(x) - u(y)|^2}{|x-y|^{n+2s}} dxdy  - \frac{1}{2 \cdot 2^{*}_{\mu}} \int_{\Omega} \int_{\Omega} \frac{|u_+(x)|^{2^{*}_{\mu}} |u_+(y)|^{2^{*}_{\mu}}}{|x-y|^{\mu}} dxdy,
    \end{equation}
and the limiting functional $J: D^{1,2}(\mathbb{R}^{n})\to \mathbb R$ associated with the problem \eqref{2} at infinity is defined by
\begin{equation}\label{Ener-Inf}
J(u) = I_{\infty}(u)=  \frac{1}{2}\int_{\mathbb{R}^{n}} |\nabla u|^2 \,{\rm d}x  -\frac{1}{2\cdot2_\mu^*}\I{\mathbb{R}^{n}}\I{\mathbb{R}^{n}}\frac{|u_+(x)|^{2_\mu^*}|u_+(y)|^{2_\mu^*}}{|x-y|^\mu}dxdy.   
\end{equation}
Using the classical Hardy-Littlewood-Sobolev inequality
(see Proposition~\ref{HLS}), we observe that
\[
I \in C^1\bigl(\mathcal{X}^{1,2}_0(\Omega), \mathbb{R}\bigr),
\]
and for any $u,\varphi \in \mathcal{X}^{1,2}_0(\Omega)$ one has
\begin{align*}
	\ld I^\prime(u), \varphi \rd =& \int_{\mathbb{R}^n} \na u\cdot\na \varphi \,dx + \int_{\mathbb{R}^{2n}}\frac{(u(x)-u(y))(\varphi(x)-\varphi(y))}{|x-y|^{n+2s}}dxdy\\
	&-\I{\Om}\I{\Om}\frac{|u_+(y)|^{2_\mu^*}|u_+(x)|^{2_\mu^*-1}\varphi(x)}{|x-y|^\mu}dxdy.
\end{align*} 
Note that if $u\in \X (\Omega)$ is a weak solution of~\eqref{2}, then $u$ is a critical point of $I$ and vice versa. One of the standard tools in variational methods is to find a minimizer of the corresponding energy functional as the limit of minimizing sequences.

The following definition gives the desired compactness condition for the extremizing sequences associated to the energy functional.

\begin{Definition}\label{PS-def}
    A sequence $\{ u_k \} \subset \X(\Omega)$ is said to be a Palais-Smale $(PS)$ sequence for $I$ at level $\theta$, if $I(u_k) \to \theta$ in $\mathbb{R}$ and $I'(u_k) \to 0$ in $(\X(\Omega))'$ as $k \to \infty$. The function $I$ is said to satisfy $(PS)$ condition at level $\theta$, if every $(PS)$ sequence for $I$ at level $\theta$ admits a convergent subsequence.
\end{Definition}
%      We first recall a useful auxiliary lemma from \cite{giovanni}. 
% \begin{Lemma}\label{LeoniIneq}
% Let $\Omega \subseteq \mathbb{R}^{n}$ be an open set, $0<s_1<s_2<1$. Then $\forall u \in H^{s_2}(\Omega)$, there exists a constant $c>0$ such that 
% \begin{equation*}
%     [u]_{s_1}\leq c \|u\|_{L^2(\Omega)}^{1-s_1/s_2}[u]_{s_2}^{s_1/s_2}.
% \end{equation*}
% \end{Lemma}

The next compactness result, whose proof can be found in [\cite{chakraborty2025global}, Lemma 2.2], shall be useful in the proof of our main global compactness Theorem \ref{PS_decomposition}.
    \begin{Lemma}\label{compact_embedding}
Let $\Omega \subset \mathbb{R}^n$ be a bounded and nonempty open set. Then, the embedding $\mathcal{X}_0^{1,2}(\Omega) \hookrightarrow H^s(\Omega)$ is compact.
\end{Lemma}
We use $S_{H,L,C}$ to denote the best constant associated to
 	\begin{align}\label{ebc}	S_{H,L,C}=\inf\limits_{u \in C_0^\infty(\mb R^n)\setminus \{0\}} \frac{\|\nabla u\|_{L^2(\mb R^n)}^2}{\left(\I{\mb R^n}\I{\mb R^n}\frac{|u_+(x)|^{2_\mu^\ast}|u_+(y)|^{2_\mu^\ast}}{|x-y|^\mu}dxdy\right)^\frac{1}{2_\mu^\ast}}.
 	\end{align}
    Since problems involving the Choquard nonlinearity are closely related to the Hardy-Littlewood-Sobolev inequality, we begin by recalling the inequality in the context of mixed operators. In \cite{anthal2023choquard}, the author considered an open set $\Omega \subset \mathbb{R}^n$(not necessarily bounded or connected) and functions $u: \mathbb{R}^n \to \mathbb{R}$ that vanish outside $\Omega$. For such functions, a mixed Hardy-Littlewood-Sobolev inequality of the form
\begin{equation}\label{eq:mixedHLS}
S_{H,L,M}(\Omega) \int_{\Omega} \int_{\Omega} \frac{|u_+(x)|^{2_\mu^\ast} |u_+(y)|^{2_\mu^\ast}}{|x-y|^\mu} \, dx \, dy 
\leq \|\nabla u\|_{L^2(\mathbb{R}^n)}^2 +  \int_{\mathbb{R}^{2n}} \frac{|u(x)-u(y)|^2}{|x-y|^{n+2s}} \, dx \, dy
\end{equation}
holds. Here, $S_{H,L,M}(\Omega)$ denotes the largest constant for which this inequality is valid. From \cite[Theorem 1.2]{anthal2023choquard}, one has
\[
S_{H,L,M}(\Omega) =S_{H,L,C}
\]
where $n \geq 3$ and constant $S_{H,L,C}$ is defined by \eqref{ebc}. 
For an open set  $\Omega \subset \mathbb{R}^n$, the best constant in the classical Hardy-Littlewood-Sobolev inequality is defined by
\begin{equation}\label{chc}
S_{H,L,C}(\Omega) = \inf_{u \in C_0^\infty(\Omega) \setminus \{0\}} 
\frac{\|\nabla u\|_{L^2(\Omega)}^2}
{\left( \int_{\Omega} \int_{ \Omega} \frac{|u_+(x)|^{2_\mu^\ast} |u_+(y)|^{2_\mu^\ast}}{|x-y|^\mu} \, dx \, dy \right)^{\frac{1}{2_\mu^\ast}}}.
\end{equation}
% Associated with this inequality, one can define the following Rayleigh-type quotient (taking $C(n,s) = 2$)
% \begin{equation}\label{eq:rayleigh}
% S(u;\Omega) := \frac{\|\nabla u\|_{L^2(\Omega)}^2 + [u]_s^2}{\int_{\Omega}\int_{ \Omega} \frac{|u(x)|^{2_\mu^\ast} |u(y)|^{2_\mu^\ast}}{|x-y|^\mu} \, dx \, dy}, \quad u \in \mathcal{C}_0^\infty(\Omega),
% \end{equation}
% where $[u]_s$ denotes the Gagliardo seminorm
% \[
% [u]_s^2 := \int_{\mathbb{R}^{2n} } \frac{|u(x)-u(y)|^2}{|x-y|^{n+2s}} \, dx \, dy,
% \]
% and $\langle \cdot, \cdot \rangle_s$ represents the inner product associated with this seminorm.

In \cite{anthal2023choquard}, the authors studied the Choquard problem for mixed local-nonlocal operators and, among other results, established the following theorem.

\begin{Theorem}\label{thm:best_constant}
For $s \in (0,1)$ and any open set $\Omega \subseteq \mathbb{R}^n$, the best Sobolev constant $\mathcal{S}_{H,L,M}(\Omega)$ in the inequality \eqref{eq:mixedHLS} coincides with $S_{H,L,C}$, the optimal constant in the Hardy–Littlewood–Sobolev inequality defined in \eqref{ebc}. Moreover, $\mathcal{S}_{H,L,M}(\Omega)$ is never attained.
\end{Theorem}

We briefly recall some fundamental properties of the best constant $S_{H,L,C}$. 
For detailed proofs and further discussions, we refer the reader to \cite{anthal2023choquard}. 

\begin{Proposition}\label{r3.1}
Let $n \ge 3$ and $\Omega \subset \mathbb{R}^n$ be an open set. Then
\begin{enumerate}
    \item It holds that
    \[
    S_{H,L,C} = S_{H,L,C}(\Omega).
    \]
   % where the definition of $S_{H,L,C}(\mathbb{R}^n)$ is given in~\eqref{ebc}.

    \item The best constant $S_{H,L,C}(\Omega)$ is not attained unless $\Omega = \mathbb{R}^n$.

    \item When $\Omega = \mathbb{R}^n$, the constant $S_{H,L,C}$ is achieved by the family of functions
    \[
    \mathcal{A} = \left\{ 
    V_{t,x_0}(x) = t^{\frac{2-n}{2}} \, \mathcal{U}\!\left( \frac{x - x_0}{t} \right) 
    : t > 0,~ x_0 \in \mathbb{R}^n 
    \right\},
    \]
    where 
    \begin{equation}\label{U}
    \mathcal{U}(y) := c \left( 1 + |y|^2 \right)^{\frac{2-n}{2}}.
    \end{equation}
\end{enumerate}
\end{Proposition}

By $\mathcal{L}(u)\geq 0$ in $\Omega$, we mean that
\begin{equation}\label{dd5.1}
    \int_{\Omega} \nabla u.\nabla \varphi \,dx +\int_{\mathbb{R}^{2n}} \frac{(u(x)-u(y))(\varphi(x)-\varphi(y))}{|x-y|^{n+2s}} dxdy \geq 0,
\end{equation}
for every non negative $\varphi \in \mathcal{X}^{1,2}_0(\Omega)$.
With this, we first establish the following weak maximum principle.

\begin{Proposition}\label{weakm}
Let $u \in \mathcal{X}^{1,2}_0(\Omega)$ satisfying be $\mathcal{L}u \geq 0$ in $\Omega$, then $u \geq 0$ a.e. in $\Omega$.
\end{Proposition}
\begin{proof}
We argue by contradiction. Suppose that there exists a set $X \subseteq \Omega$ of positive Lebesgue measure such that $u < 0$ a.e. on $X$, then
the support of
\(
 u_{-} = \max\{-u, 0\}> 0 \), which is a subset of $X$, has positive  measure.
% \textcolor{green}{(T.M. This part needs to be checked again!)
% Since $u \in \mathcal{X}^{1,2}_0(\Omega)$ and $u \geq 0$ a.e. in $\mathbb{R}^{n} \setminus \Omega$, it follows (see Remark \ref{remarkdef}) that
% \begin{equation}\label{eq27_new}
% w \in \mathcal{X}^{1,2}_0(\Omega) \quad \text{and} \quad w \equiv 0 \text{ a.e. in } \mathbb{R}^{n} \setminus \Omega.
% \end{equation}
% Moreover,
% \begin{equation}\label{eq28_new}
% w = -u > 0 \quad \text{on } X.
% \end{equation}}
Using the weak formulation of $\mathcal{L}u \geq 0$ in $\Omega$ and taking $\varphi=u_{-}$ as a test function in \eqref{dd5.1}, we get
    \[\int_{\Omega} \nabla u.\nabla u_{-} \,dx +\int_{\mathbb{R}{^{2n}}} \frac{(u(x)-u(y))(u_{-}(x)-u_{-}(y))}{|x-y|^{n+2s}} dxdy \geq 0,\]
    where $\int_{\Omega} \nabla u.\nabla u_{-} \,dx=-\int_{\Omega}|\nabla u_{-}|^2~dx\leq 0$ and setting $u_{+}=\max\{u,0\}$ we have
   \begin{align*}
&\int_{\mathbb{R}^{2n}} \frac{(u(x) - u(y))(u_{-}(x) - u_{-}(y))}{|x - y|^{n + 2s}} \, dx dy \\
&= \int_{\mathbb{R}^{2n}} \frac{(u_{+}(x) - u_{+}(y))(u_{-}(x) - u_{-}(y))}{|x - y|^{n + 2s}} \, dx dy 
- \int_{\mathbb{R}^{2n}} \frac{(u_{-}(x) - u_{-}(y))^2}{|x - y|^{n + 2s}} \, dx dy \\
&< \int_{\mathbb{R}{^{2n}}} \frac{(u_{+}(x) - u_{+}(y))(u_{-}(x) - u_{-}(y))}{|x - y|^{n + 2s}} \, dx dy 
\leq 0.
\end{align*}
since for each $x,y\in \mathbb R^n$, one can easily verify that
    $(u_{+}(x)-u_{+}(y))(u_{-}(x)-u_{-}(y)) \leq 0.$
This is a contradiction, and hence we conclude that $u\geq 0$ in $\Omega$ a.e.
\end{proof}
\QED

We next provide a lower bound for the energy of the solution to problem \eqref{1}, which will be crucial in the variational analysis.
\begin{Lemma}\label{1lem2.8}
Let $u \in \mathcal{X}^{1,2}_0(\Omega)$ be any solution to problem \eqref{2}. 
% \begin{equation}\label{2}
% \begin{cases}
% \mathcal{L}u(x) = 
% \left( \displaystyle\int_\Omega \frac{|u_+(y)|^{2^*_{\mu}}}{|x - y|^\mu}\,dy \right)
% |u_+(x)|^{2^*_{\mu}-1}, &  \text{in}~~ \Omega, \\[1.2em]
% ~~u(x) = 0, & ~\text{in}~~~ \mathbb{R}^n \setminus \Omega.
% \end{cases}
% \end{equation}
Then the associated energy \eqref{eqI}, satisfies
\[
I(u) \geq \frac{1}{2}\left(\frac{n-\mu+2}{2n-\mu}\right) S^{\frac{2n-\mu}{n-\mu+2}}_{H, L, C}=\beta\text{(say)},
\]
where $S_{H,L,C}$ denotes the best constant defined by \eqref{ebc}. Moreover, the same conclusion holds for solutions 
$u \in \mathcal{X}^{1,2}_0(\mathbb{R}^{n})$ to 
\[
\mathcal{L}u = 
\left( \int_{\mathbb{R}^n} \frac{|u_+(y)|^{2^*_{\mu}}}{|x - y|^\mu} \, dy \right)
|u_+(x)|^{2^*_{\mu}-1}
\quad \text{in } \mathbb{R}^n.
\]
\end{Lemma}
%%%%%%%%%%%%%%%%%%%%

\begin{proof}
Suppose $u$ is a weak solution of \eqref{2}, then $u=u_+$ by virtue of Remark \ref{r2p}.
Now, $u_+$ is taken as a test function in the weak formulation of \eqref{2} to obtain
\[
\int_{\Omega} \nabla u \cdot \nabla u_+ \, dx 
+ \int_{\mathbb{R}^{2n}} \frac{(u(x) - u(y))(u_+(x) - u_+(y))}{|x - y|^{n + 2s}} \, dx dy 
= \int_{\Omega} \int_{\Omega}\frac{|u_+(y)|^{2_{\mu}^*}|u_+(x)|^{2_{\mu}^*-1} u_{+}(x)}{|x-y|^\mu}dxdy.
\]
Expanding and simplifying the local and nonlocal terms yield
\begin{equation}\label{plusTestFunc}
\int_{\Omega} |\nabla u_+|^2 \, dx 
+ \int_{\mathbb{R}^{2n}} \frac{|u_+(x) - u_+(y)|^2}{|x - y|^{n + 2s}} \, dx dy 
+ 2 \int_{\mathbb{R}^{2n}} \frac{u_+(x) u_-(y)}{|x - y|^{n + 2s}} \, dx dy 
= \int_{\Omega} \int_{\Omega}\frac{|u_+(y)|^{2_{\mu}^*}|u_+(x)|^{2_{\mu}^*} }{|x-y|^\mu}dxdy.
\end{equation}
% Similarly, $u_{-}$ is taken as a test function in the weak formulation of \eqref{2}. Then we obtain
% \begin{equation}
% -\int_{\Omega} |\nabla u_-|^2 \, dx 
% - \int_{\mathbb{R}^{2n}} \frac{|u_-(x) - u_-(y)|^2}{|x - y|^{n + 2s}} \, dx dy 
% - 2 \int_{\mathbb{R}^{2n}} \frac{u_+(x) u_-(y)}{|x - y|^{n + 2s}} \, dx dy 
% = 0.
% \end{equation}
% Now, it is easy to see that $\eta(u_{-})=0$ that implies $u_{-} =0$ a.e. in $\mathbb{R}^n$. 
So, from \eqref{plusTestFunc} and using remark \ref{r2p}, we have 
\[
\eta(u_+)^2
=
\int_{\Omega} \int_{\Omega}\frac{|u_+(x)|^{2_{\mu}^*}|u_+(y)|^{2_{\mu}^*}}{|x-y|^\mu}dxdy.
\]
To estimate $\eta(u_+)^2$, we use inequality \eqref{eq:mixedHLS} and conclude that
\[
(S_{H,L,M})^{\frac{2^*_{\mu}}{2^*_{\mu}-1}} (\Omega)\leq 
\int_{\Omega}\int_{\Omega}\frac{|u_+(x)|^{2_{\mu}^*}|u_+(y)|^{2_{\mu}^*}}{|x-y|^\mu}dxdy
=\eta(u+)^2=\eta(u)^2.
\]
Now we simplify the coefficient using $2_{\mu}^*=\dfrac{2n-\mu}{n-2}$ to get
\[
\frac12 - \frac{1}{2\cdot 2_{\mu}^*}
= \frac12\left(\frac{n - \mu + 2}{2n - \mu}\right).
\]
Finally, we can deduce that
\[
I(u) \ge \frac12\left(\frac{n - \mu + 2}{2n - \mu}\right) S_{H,L,M}^{\frac{2n - \mu}{n - \mu + 2}}(\Omega)\ge \frac12\left(\frac{n - \mu + 2}{2n - \mu}\right) S_{H,L,C}^{\frac{2n - \mu}{n - \mu + 2}}.
\]
while using the fact that $S_{H,L,M}(\Omega) = S_{H,L,C}$.
A similar reasoning applies even when $\Omega=\mathbb{R}^n$, completing the proof. 
\end{proof}\QED

%%%%%%%%%%%%%%%%%%%%%%%%%%%%%%%%%%%%%%%%%%%%
\section{Regularity of weak solutions}

In this section, we establish that any weak solution of \eqref{1} is bounded and, furthermore, belongs to the space $C^{2,\beta}(\overline{\Omega})$ for a suitable $\beta>0$.

\textbf{Proof of Theorem \ref{thm:7.1}.}
We may assume that $u\not\equiv0$ and $u\ge0$ a.e. in $\Omega$.
The argument follows the De Giorgi iteration scheme as in \cite[Theorem~6.2]{mukherjee2016fractional} and also one can refer to the details in Theorem 1.5 of \cite{anthal2023choquard}. Let $m>0$ and define
\[
\mathcal U(x)=\frac{u(x)}{m}.
\]
such that $\|\mathcal U(x)\|_{L^{2^*_\mu}(\mathbb R^n)}=\rho$, for some $\rho>0$. Then $\mathcal U\in\mathcal X^{1,2}_0(\Omega)$ and satisfies
\begin{equation}\label{eq55}
\int_{\mathbb{R}^{n}}\nabla\mathcal U\cdot\nabla\varphi\,dx
+\int_{\mathbb{R}^{2n}}\frac{(\mathcal U(x)-\mathcal U(y))(\varphi(x)-\varphi(y))}{|x-y|^{n+2s}}\,dx\,dy
\le
\int_{\Omega}\int_{\Omega}
\frac{|\mathcal U(y)|^{2^*_\mu}|\mathcal U(x)|^{2^*_\mu-1}\varphi(x)}{|x-y|^\mu}\,dy\,dx
\end{equation}
for every $0\le\varphi\in\mathcal X^{1,2}_0(\Omega)$.
For $k\in\mathbb N$, set
\[
\mathcal Z_k=1-\frac1{2^k},\qquad
\mathcal U_k=\mathcal U-\mathcal Z_k,\qquad
v_k=(\mathcal U_k)^+,
\qquad
\bar V_k=\|v_k\|_{L^{2^*}(\Omega)} .
\]
Testing \eqref{eq55} with $\varphi=v_{k+1}$ and using the elementary inequality
\[
|v_k(x)-v_k(y)|^2
\le (\mathcal U_k(x)-\mathcal U_k(y))(v_k(x)-v_k(y)),
\]
together with the Hardy-Littlewood-Sobolev inequality, we obtain the estimate
\[
\eta(v_{k+1})^2
\le C\,2^{\alpha(k+1)}
\Big(\|v_k\|_{L^{2^*}(\Omega)}^{2\cdot2^*_\mu}
+
|\{v_{k+1}>0\}|^{\frac{\mu}{2n}}
\|v_k\|_{L^{2^*}(\Omega)}^{2^*_\mu}\Big),
\]
for some constants $C>0$ and $\alpha>0$ independent of $k$.
Using the mixed Sobolev inequality and the inclusion,
\[
\{v_{k+1}>0\}\subset
\left\{v_k>\frac1{2^{k+1}}\right\},
\]
we deduce the recursive estimate
\begin{equation}\label{eq520}
\bar V_{k+1}
\le
\mathcal B^{k+1}
\left(
\bar V_k^{2^*_\mu}
+
\bar V_k^{\frac{2^*}{2}}
\right),
\end{equation}
for some constant $\mathcal B>1$.
Choosing $\rho>0$ sufficiently small and using a standard iteration argument, we obtain
\[
\bar V_k\le a\,t^{k+1},
\qquad t\in(0,1),
\]
for some $a>0$. Hence, $\bar V_k\to0$ as $k\to\infty$. By the dominated convergence theorem,
\[
\lim_{k\to\infty}\bar V_k
=
\left(\int_\Omega[(\mathcal U-1)^+]^{2^*}dx\right)^{\frac1{2^*}}
=0,
\]
which implies $\mathcal U\le1$ a.e. in $\Omega$. Therefore $u\in L^\infty(\Omega)$.
Since $u\equiv0$ in $\mathbb R^n\setminus\Omega$, we conclude that
$
u\in L^\infty(\mathbb R^n).
$
\QED

\textbf{Proof of Theorem \ref{thm1.5}:}
From Theorem~\eqref{thm:7.1}, we already know that $u \in L^{\infty}(\mathbb{R}^n)$. 
We now establish that $u \in C^{1,\alpha}(\overline{\Omega})$ for some $\alpha \in (0,1)$. Since $0<\mu<n$ and $\Omega$ is bounded, the right-hand side of \eqref{1} belongs to $L^{\infty}(\Omega)$.
% we have
% \begin{align}
% \left|\int\limits_{\Omega} \frac{|u(y)|^{2_{\mu}^{*}}}{|x-y|^{\mu}} \, dy \right|
% &\leq \|u\|_{L^{\infty}(\Omega)}^{2_{\mu}^{*}}
% \Bigg[
% \int\limits_{\Omega \cap \{|x-y|<1\}} \frac{dy}{|x-y|^{\mu}}
% + \int\limits_{\Omega \cap \{|x-y|\geq 1\}} \frac{dy}{|x-y|^{\mu}}
% \Bigg] \notag\\
% &\leq \|u\|_{L^{\infty}(\Omega)}^{2_{\mu}^{*}}
% \Bigg[
% \int\limits_{\Omega \cap \{r<1\}} r^{n-1-\mu} \, dr
% + |\Omega|
% \Bigg] <+ \infty.
% \end{align}
If $s \in (0,1/2]$, we apply \cite[Theorem~1.4]{su2022regularity} to obtain 
$u \in W^{2,p}(\Omega)$ for all $p \in (1,\infty)$, 
which implies $u \in C^{1,\beta}(\overline{\Omega})$ for every $\beta \in (0,1)$. 
If instead $s \in (1/2,1)$, then by \cite[Theorem~2.7]{biagi2023faber}, 
we have $u \in W^{2,p}(\Omega)$ for all $p \in \left(1,\frac{n}{2s-1}\right)$,
and consequently $u \in C^{1,\beta}(\overline{\Omega})$ for $\beta \in (0,2-2s)$.
Next, we study the regularity of the Riesz potential term
\[
w(x)=\int\limits_{\Omega}\frac{|u(y)|^{2_{\mu}^{*}}}{|x-y|^{\mu}}\,dy.
\]
Since $u\in L^{\infty}(\mathbb{R}^{n})$ and $u \equiv 0$ in $\mathbb{R}^{n}\setminus\Omega$, 
we deduce $|u|^{2_{\mu}^{*}},\,w \in L^{\infty}(\Omega)$. 
Let $d \in (0,n)$ be such that $\mu = n - d$. 
Then, applying \cite[Proposition~1.4(iii)]{ros2014extremal}, 
we infer that $w \in C^{0,\beta}(\overline{\Omega})$ for all $\beta < \min\{s,d\}$.
Therefore, the right-hand side of \eqref{1} belongs to $C^{0,\beta}(\overline{\Omega})$.
Finally, assuming that $\partial\Omega$ is of class $C^{2,\alpha}$, 
an application of \cite[Theorem~2.8]{biagi2023faber} yields 
$u \in C^{2,\beta}(\overline{\Omega})$ for some 
$\beta \leq \min\{b,\alpha\}$, $b>0$
which completes the proof.
\QED

%%%%%%%%%%%%%%%%%%%%%%%%%%%%%%%%%%%%%%%%%%%%%%%%%%%%%%%%%%%%%

Now, with the aid of the regularity result, Theorem~\ref{thm1.5}, we can derive the following maximum principle.

\begin{Proposition}\label{sm}
Let $\alpha \in (0,1)$ be as in Theorem~\ref{thm1.5}, and suppose that $\partial \Omega$ is of class $C^{2,\alpha}$. 
Assume that $u$ is a nontrivial nonnegative solution 
of~\eqref{1}. Then $u>0$ in $\Omega$.
\end{Proposition}

\begin{proof}
Assume, by contradiction, that there exists $x_0 \in \Omega$ such that $u(x_0)=0$. 
Since $u \geq 0$ in $\Omega$, it follows that $x_0 = \mathrm{argmin}_{x \in {\Omega}} u(x)$. 
As $u \in C^{2,\alpha}(\bar{\Omega})$, from Theorem \ref{thm1.5}, we have $-\Delta u(x_0) \leq 0$. 
Moreover, since $u$ is nontrivial and continuous, there exists a $y_0 \in \Omega$ and $r>0$ such that $u>0$ in $B_r(y_0)$. 
This implies $(-\Delta)^s u(x_0) < 0$, and consequently $\mathcal{L}(x_0) < 0$. 
However, the nonlocal term in~\eqref{1},
% \[
% \left( \int_{\Omega} \frac{u(y)^{2_\mu^\ast}}{|x - y|^\mu} \, dy \right) u(x_0)^{2_\mu^\ast - 1},
% \]
vanishes at $x_0$, leading to a contradiction. 
Hence, $u>0$ in $\Omega$. 
\end{proof}
\QED

%%%%%%%%%%%%%%%%%%%%%%%%%%%%%%%%%%%%%%%%%%%%%
\section{Palais-Smale Sequences and Global Compactness}
This section is devoted to analysis and Palais-Smale sequences and establishes compactness results, which are crucial to proving our main results.
We start with the Brezis-Lieb type lemma for the  Choquard term. 
\begin{Lemma}\label{1lem4.2-pos}
Let $u_k \rightharpoonup u$ be weakly convergent in $\mathcal{X}^{1,2}_0(\mathbb{R}^{n})$ and $u_k \to u$ a.e.\ on $\mathbb{R}^n$. Then as $k\to \infty$
\begin{equation}\label{eq4.1-pos}
\begin{aligned}
\big(|x|^{-\mu} * |(u_k)_+|^{2^*_{\mu}}\big)\,|(u_k)_+|^{2^*_{\mu}-2}(u_k)_+ 
&- \big(|x|^{-\mu} * |(u_k - u)_+|^{2^*_{\mu}}\big)\,|(u_k - u)_+|^{2^*_{\mu}-2}(u_k - u)_+ \\
& \quad \to \big(|x|^{-\mu} * |u_+|^{2^*_{\mu}}\big)\,|u_+|^{2^*_{\mu}-2}u_+
\quad \text{in } \big( \mathcal X^{1,2}_0(\mathbb R^n))'.
\end{aligned}
\end{equation}
\end{Lemma}

\begin{proof}
Since $u_k \rightharpoonup u$ weakly in $\mathcal{X}^{1,2}_0(\mathbb{R}^{n})$, there exists $M > 0$ such that $\eta(u_k) < M$ for all $k \in \mathbb{N}$. Let {$\phi \in \mathcal{X}^{1,2}_0(\mathbb{R}^{n})$} and define
\[
I_k = \int_{\mathbb{R}^n} \bigl[ \big(|x|^{-\mu} * |(u_k)_+|^{2^*_{\mu}}\big)\,|(u_k)_+|^{2^*_{\mu}-2}(u_k)_+ - \big(|x|^{-\mu} * |(u_k - u)_+|^{2^*_{\mu}}\big)\,|(u_k - u)_+|^{2^*_{\mu}-2}(u_k - u)_+ \bigr]\phi\, dx.
\]
Then $I = I_1 + I_2 + I_3 - 2 I_4$, where
\begin{align*}
I_1 &= \int_{\mathbb{R}^n}
\Bigl(
(|x|^{-\mu} * |(u_k)_+|^{2^*_{\mu}})
-
(|x|^{-\mu} * |(u_k - u)_+|^{2^*_{\mu}})
\Bigr) \\
&\quad \times
\Bigl[
|(u_k)_+|^{2^*_{\mu}-2}(u_k)_+
-
|(u_k - u)_+|^{2^*_{\mu}-2}(u_k - u)_+
\Bigr] \phi\, dx, \\[6pt]
I_2 &= \int_{\mathbb{R}^n}
(|x|^{-\mu} * |(u_k)_+|^{2^*_{\mu}})
\, |(u_k - u)_+|^{2^*_{\mu}-2}(u_k - u)_+
\, \phi\, dx, \\[6pt]
I_3 &= \int_{\mathbb{R}^n}
(|x|^{-\mu} * |(u_k - u)_+|^{2^*_{\mu}})
\, |(u_k)_+|^{2^*_{\mu}-2}(u_k)_+
\, \phi\, dx, \\[6pt]
I_4 &= \int_{\mathbb{R}^n}
(|x|^{-\mu} * |(u_k - u)_+|^{2^*_{\mu}})
\, |(u_k - u)_+|^{2^*_{\mu}-2}(u_k - u)_+
\, \phi\, dx.
\end{align*}
\noindent\textbf{Claim 1.} It holds that
\[
\lim_{k \to \infty} I_1 = \int_{\mathbb{R}^n} (|x|^{-\mu} * |u_+|^{2^*_{\mu}})\,|u_+|^{2^*_{\mu}-2}u_+ \phi dx.
\]
% Since $\{u_k\}$ is a bounded sequence in $\mathcal{X}^{1,2}_0(\mathbb{R}^{n})$ and by the Sobolev embedding,
% \(
% \mathcal{X}^{1,2}_0(\mathbb{R}^{n}) \hookrightarrow L^{2^*}(\mathbb{R}^n),
% \)
% there exists a constant $C>0$ such that
% \[
% \|u_k\|_{L^{2^*}(\mathbb{R}^n)} \le C,~\forall~ k\in\mathbb{N}.
% \]
% For $p = 2_{\mu}^* = \frac{2n-\mu}{n-2}$, note that $p < 2^*$, and therefore $\{u_k\}$ is also bounded in $L^p(\mathbb{R}^n)$, since $\mathcal{X}^{1,2}_0(\mathbb{R}^{n}) \hookrightarrow L^r(\mathbb{R}^n), ~r \in [2, 2^*]$.
% Consequently, 
% \[
% |(u_k)_+|^{2^*_{\mu}} \in L^{\frac{2n}{2n-\mu}}(\mathbb{R}^n), ~\text{for all } k.
% \]
  {First, similarly to the proof of the Brezis-Lieb lemma \cite{brezis1983relation}, we know that
\[
|(u_k)_+|^{2^*_{\mu}} - |(u_k - u)_+|^{2^*_{\mu}} \to |u_+|^{2^*_{\mu}} \quad \text{in } L^{\frac{2n}{2n - \mu}}(\mathbb{R}^n), \quad \text{as } k \to \infty.
\]
Then, using the Hardy-Littlewood-Sobolev inequality, which implies that the Riesz potential defines a linear continuous map from $L^{\frac{2n}{2n-\mu}}(\mathbb{R}^{n})$ to $L^{\frac{2n}{\mu}}(\mathbb{R}^{n})$, we deduce that
\begin{equation}\label{eq4.2-pos}
|x|^{-\mu} * \left( |(u_k)_+|^{2^*_{\mu}} - |(u_k - u)_+|^{2^*_{\mu}} \right) \to |x|^{-\mu} * |u_+|^{2^*_{\mu}} \quad \text{strongly in } L^{\frac{2n}{\mu}}(\mathbb{R}^n) ~~\text{as}~k\to\infty.
\end{equation}}
Since both $|(u_k)_+|^{2^*_{\mu}-2}(u_k)_+ \rightharpoonup |u_+|^{2^*_{\mu}-2}u_+$ and $|(u_k - u)_+|^{2^*_{\mu}-2}(u_k - u)_+ \rightharpoonup 0$ weakly in $L^{\frac{2n}{2n-\mu}}(\mathbb{R}^n)$,
we obtain
\begin{equation}\label{eq4.3-pos}
|(u_k)_+|^{2^*_{\mu}-2}(u_k)_+  - |(u_k - u)_+|^{2^*_{\mu}-2}(u_k - u)_+ \rightharpoonup |u_+|^{2^*_{\mu}-2}u_+ \quad \text{weakly in } L^{\frac{2n}{2n - \mu}}(\mathbb{R}^{n}).
\end{equation}
Thus, combining \eqref{eq4.2-pos} and \eqref{eq4.3-pos}, Claim 1 follows.

\noindent\textbf{Claim 2.} It holds \[\lim_{k \to \infty} I_2 = 0.\]
Since $|(u_k)_+|^{2^*_{\mu}} \rightharpoonup |u_+|^{2^*_{\mu}}$ weakly in $L^{\frac{2n}{2n-\mu}}(\mathbb{R}^n)$(up to a subsequence), by the Hardy-Littlewood-Sobolev inequality,
\begin{equation}\label{eq4.4-pos}
|x|^{-\mu} * |(u_k)_+|^{2^*_{\mu}} \rightharpoonup |x|^{-\mu} * |u_+|^{2^*_{\mu}} \quad \text{weakly in } L^{\frac{2n}{\mu}}(\mathbb{R}^n).
\end{equation}
Next, we aim to show that 
\[
|(u_k - u)_+|^{2^*_{\mu}-2}(u_k - u)_+\to 0 
\quad \text{strongly in } L^{\frac{2n}{2n-\mu}}(\mathbb{R}^n).
\]
Indeed, since $u_k \to u$ a.e. in $\mathbb{R}^{n}$, we have
\[
|(u_k - u)_+|^{2^*_{\mu}-2}(u_k - u)_+\to 0 
\quad \text{a.e. in } \mathbb{R}^n.
\]
Moreover, for any open set $U \subset \mathbb{R}^n$, H\"older’s inequality gives
\begin{align*}
\int_U \left| |(u_k - u)_+|^{2^*_{\mu}-2}(u_k - u)_+\phi \right|^{\frac{2n}{2n-\mu}} dx 
&\leq \left(\int_U |(u_k - u)_+|^{2^*_{\mu}} dx\right)^{\frac{n-\mu+2}{2n-\mu}} \left(\int_U |\phi|^{2^*_{\mu}} dx\right)^{\frac{n-2}{2n-\mu}} \\
&\leq M \left(\int_U |\phi|^{2^*_{\mu}} dx\right)^{\frac{n-2}{2n-\mu}}.
\end{align*}
This implies $\{||(u_k - u)_+|^{2^*_{\mu}-2}(u_k - u)_+\phi|^{\frac{2n}{2n-\mu}}\}_k$ is equi-integrable in $L^1(\mathbb{R}^n)$.  
So by Vitali's convergence theorem, we get
\[
|(u_k - u)_+|^{2^*_{\mu}-2}(u_k - u)_+\phi \to 0 \quad \text{strongly in } L^{\frac{2n}{2n-\mu}}(\mathbb{R}^n).
\]
Finally, combining \eqref{eq4.4-pos} with the above strong convergence, we conclude that $
\lim_{k\to\infty}I_2 = 0$
which completes Claim 2. Next, we can prove that
\[
\lim_{k\to\infty} I_3 = 0 \qquad\text{and}\qquad \lim_{k\to\infty} I_4 = 0.
\]
using arguments identical to those used in Claim~2.
Consequently,  \eqref{eq4.1-pos} holds that is
\[
I_k \to \int_{\mathbb{R}^n} \big(|x|^{-\mu} * |u|^{2^*_{\mu}}\big) |u_+|^{2^*_{\mu}-2} u_+\phi \, dx.
\]
\end{proof}
\QED
Our next lemma talks about the characterization of (PS)-sequences with respect to energy functionals $I$ and $J$.
\begin{Lemma}\label{lem4.2}
Let \(\{u_k\} \subset \mathcal{X}^{1,2}_0(\Omega)\) be a (PS)-sequence such that  
\[
u_k \rightharpoonup u \quad \text{weakly in } \mathcal{X}^{1,2}_0(\Omega), \quad u_k \to u \quad \text{a.e.\ in } \Omega,
\]
and suppose that 
\(
I(u_k) \to c, ~ \text{and}~ I'(u_k) \to 0 ~ \text{in } (\mathcal{X}^{1,2}_0(\Omega))',
\)
as \(k \to \infty\). Then the following holds:  
\begin{enumerate}
    \item \(I'(u) = 0\) in \((\mathcal{X}^{1,2}_0(\Omega))'\) and defining \(v_k := u_k - u\), it satisfies as $k\to \infty$
    \[
    \eta(v_k)^2 = \eta(u_k)^2 - \eta(u)^2 + o(1).
    \] 
    \item The sequence $\{v_k\}$ satisfies, as $k\to \infty$,
    \[
    J(v_k) \to c - I(u) ~~~\text{and}~~ J'(v_k) \to 0 \quad \text{in } (\mathcal{X}^{1,2}_0(\Omega))'.
    \]
\end{enumerate}
\end{Lemma}
\begin{proof}
From the given hypothesis,
\[
I(u_k) \to c \quad \text{and} \quad I'(u_k) \to 0 \quad \text{in } \left(\mathcal{X}^{1,2}_0(\Omega)\right)' \ \text{as } k \to \infty,
\]
%it is straightforward to verify that the sequence $\{u_k\}$ is bounded in $\mathcal{X}^{1,2}_0(\Omega)$. 
The convergence
\(
u_k \rightharpoonup u ~\text{weakly in } \mathcal{X}^{1,2}_0(\Omega)
\)
 implies that, for every test function $\varphi \in \mathcal{X}^{1,2}_0(\Omega)$, we have 
\begin{equation}\label{eqwe}
\begin{aligned}
\langle u_k, \varphi \rangle_{\eta} 
&:= \int_{\mathbb{R}^n} \nabla u_k \cdot \nabla \varphi \, \mathrm{d}x 
+ \int_{\mathbb{R}^{2n}} \frac{(u_k(x) - u_k(y))(\varphi(x) - \varphi(y))}{|x-y|^{n+2s}} \, \mathrm{d}x \, \mathrm{d}y \\
&\longrightarrow \langle u, \varphi \rangle_{\eta} 
= \int_{\mathbb{R}^n} \nabla u \cdot \nabla \varphi \, \mathrm{d}x 
+ \int_{\mathbb{R}^{2n}} \frac{(u(x) - u(y))(\varphi(x) - \varphi(y))}{|x-y|^{n+2s}} \, \mathrm{d}x \, \mathrm{d}y, ~~\text{as} ~k \to \infty.
\end{aligned}
\end{equation}
Now, we claim that $I'(u) = 0$.
Note that
$
u_k \rightharpoonup u \ \text{weakly in } \mathcal{X}^{1,2}_0(\Omega)$ and convergence a.e. in $\Omega$ implies that $u_k \rightharpoonup u$ in $L^{2^*}(\Omega)$,
% as $k \to +\infty$ impl
\[
(u_k)_+^{2^*_{\mu}} \rightharpoonup u_+^{2^*_{\mu}} \quad \text{in } L^{\frac{2n}{2n-\mu}}(\Omega)
\]
and
\[
(u_k)_+^{2^*_{\mu}-1} \rightharpoonup u_+^{2^*_{\mu}-1} \quad \text{in } L^{\frac{2n}{n+2-\mu}}(\Omega)
\]
as $k \to +\infty$. Since the Riesz potential defines a continuous map from $L^{\frac{2n}{2n-\mu}}(\Omega)$ to $L^{\frac{2n}{\mu}}(\Omega)$, using the Hardy-Littlewood-Sobolev inequality together with a.e. convergence, we get
\[
\int_\Omega \frac{|(u_k)_+(y)|^{2^*_{\mu}}}{|x-y|^\mu}\,dy
\rightharpoonup
\int_\Omega \frac{|u_+(y)|^{2^*_{\mu}}}{|x-y|^\mu}\,dy
\quad \text{in } L^{\frac{2n}{\mu}}(\Omega)
\]
as $k \to +\infty$. Using these facts together, we obtain
\[
\int_\Omega \frac{((u_k)_+(y))^{2^*_{\mu}}((u_k)_+(x))^{2^*_{\mu}-1}}{|x-y|^\mu}\,dy
\rightharpoonup
\int_\Omega \frac{(u_+(y))^{2^*_{\mu}}(u_+(x))^{2^*_{\mu}-1}}{|x-y|^\mu}\,dy
\quad \text{in } L^{\frac{2n}{n+2}}(\Omega)
\]
as $k \to +\infty$. This implies for any $\varphi \in \mathcal{X}^{1,2}_0(\Omega)$, we have
\begin{equation}\label{eq:4.5}
\int_{\Omega } \int_{\Omega } \frac{ |(u_k)_+(y)|^{2^*_\mu} \varphi(x)\, |u_k(x)|^{2^*_\mu-2}(u_k)_+(x) }{ |x - y|^\mu }\, dx dy 
\to
\int_{\Omega }\int_{\Omega }  \frac{ |u_+(y)|^{2^*_\mu} |u_+(x)|^{2^*_\mu-2}u_+(x) \, \varphi(x)}{ |x - y|^\mu }\, dx dy. 
\end{equation}
Now, using \eqref{eqwe}- \eqref{eq:4.5} and for $\varphi \in \mathcal{X}^{1,2}_0(\Omega)$, we obtain
\begin{align*}
\langle I'(u_k) - I'(u), \, \varphi \rangle 
&= \int_{\mathbb{R}^n} \nabla u_k\cdot \nabla \varphi\,dx+ \int_{\mathbb{R}^{2n}}\frac{(u_k(x)-u_k(y))(\varphi(x)-\varphi(y))}{|x-y|^{n+2s}}\,dx\,dy \\
&\quad - \int_{\Omega} \int_{\Omega } \frac{ |(u_k)_+(x)|^{2^*_{\mu}-2} \, |(u_k)_+(y)|^{2^*_{\mu}} (u_k)_+ \varphi(x)}{ |x - y|^\mu }\, dx dy \\
&\quad - \int_{\mathbb{R}^n} \nabla u\cdot \nabla \varphi\,dx - \int_{\mathbb{R}^{2n}} \frac{(u(x)-u(y))(\varphi(x)-\varphi(x))}{|x-y|^{n+2s}}\,dx\,dy \\
&\quad + \int_{\Omega}\int_{\Omega }  \frac{ |u_+(x)|^{2^*_{\mu}-2} \, |u_+(y)|^{2^*_{\mu}} (u_k)_+\varphi(x)}{ |x - y|^\mu }\, dx dy.
\end{align*}
 By \eqref{eq:4.5} and the fact that $u_k \rightharpoonup u$ weakly in $\mathcal{X}^{1,2}_0(\Omega)$, the claim follows. Moreover, $u$ is a weak solution of \eqref{2}. 

\noindent Now, we define $v_k = u_k - u$, then it is easy to see by the Brezis-Lieb Lemma that the following holds:
    \[
    \eta(v_k)^2 = \eta(u_k)^2 - \eta(u)^2 + o(1).
    \] Now, we aim to show that the sequence $\{v_k\}$ forms a $(PS)$ sequence for the limiting functional $J$ at the level $c - I(u)$.
Since the embedding $\mathcal{X}^{1,2}_0(\Omega) \hookrightarrow H_0^1(\Omega)$ is continuous and $u_k \rightharpoonup u$ in $\mathcal{X}^{1,2}_0(\Omega)$, it follows that $v_k \rightharpoonup 0$ in $H_0^1(\Omega)$. This yields the following identity using the Brezis-Lieb lemma 
\begin{equation}\label{eq3.6}
    \int_{\mathbb{R}^n} |\nabla v_k|^2 
    = \int_{\Omega} |\nabla v_k|^2 
    = \int_{\Omega} |\nabla u_k|^2 - \int_{\Omega} |\nabla u|^2 + o(1), 
   \end{equation}
where $o(1)\rightarrow 0$ as $k \to \infty$.
\smallskip
In addition, employing Lemma \ref{compact_embedding} along with the Brezis-Lieb Lemma, we obtain
\begin{equation}\label{gagli}
    o(1)= [v_k]_s^{2} = [u_k]_s^{2} - [u]_s^{2} + o(1)
\end{equation}
Now, by the Brezis-Lieb lemma \cite[Lemma 2.3]{gao2016brezis}, using \eqref{gagli} \eqref{eq3.6} and Lemma \ref{1lem4.2-pos} we get the following: 
\begin{align*}
J(v_k) &= \frac{1}{2}\int_{\mathbb{R}^n} |\nabla v_k|^2\,dx 
-\frac{1}{2\cdot 2^*_{\mu}} \int_{\Omega}\int_{\Omega}  
 \frac{ |(v_k)_+(x)|^{2^*_{\mu}} |(v_k)_+(y)|^{2^*_{\mu}} }{ |x - y|^\mu }\, dx\, dy \\
&= \frac{1}{2}\int_{\mathbb{R}^n} |\nabla u_k|^2\, dx  
+ \frac{1}{2}\int_{\mathbb{R}^{2n}} \frac{|u_k(x)-u_k(y)|^2}{|x-y|^{n+2s}}\, dx\, dy  
-\frac{1}{2\cdot 2^*_{\mu}} \int_{\Omega}\int_{\Omega}
 \frac{ |(u_k)_+(x)|^{2^*_{\mu}} |(u_k)_+(y)|^{2^*_{\mu}} }{ |x - y|^\mu }\, dx\, dy \\
&\quad -\Bigg\{ \frac{1}{2}\int_{\mathbb{R}^n} |\nabla u|^2\, dx  
+ \frac{1}{2}\int_{\mathbb{R}^{2n}} \frac{|u(x)-u(y)|^2}{|x-y|^{n+2s}}\, dx\, dy  
-\frac{1}{2\cdot 2^*_{\mu}} \int_{\Omega}\int_{\Omega}
 \frac{ |u_+(x)|^{2^*_{\mu}} |u_+(y)|^{2^*_{\mu}} }{ |x - y|^\mu }\, dx\, dy \Bigg\} + o(1) \\
&= I(u_k) - I(u) + o(1) \to c - I(u), \quad \text{as } k\to\infty.
\end{align*}
Now we aim to show that $J'(v_k) \to 0$ in $(\mathcal{X}^{1,2}_0(\Omega))'$.
For $x,y \in \Omega$,  define
\[
\begin{aligned}
F_k(x,y) :=\;&
|(u_k-u)_+(x)|^{2^*_\mu - 2}\,(u_k-u)_+(x)\,
|(u_k-u)_+(y)|^{2^*_\mu} \\[1mm]
&- |(u_k)_+(x)|^{2^*_\mu - 2}\,(u_k)_+(x)\,
|(u_k)_+(y)|^{2^*_\mu} \\[1mm]
&+ |u_+(x)|^{2^*_\mu - 2}\,u_+(x)\,
|u_+(y)|^{2^*_\mu}.
\end{aligned}
\]
Then, we define the nonlocal term
\[
\mathcal{N}_k(\varphi) := \int_{\Omega} \int_{\Omega} \frac{F_k(x,y) \, \varphi(x)}{|x-y|^\mu} \, dx \, dy.
\]
 By Lemma \ref{1lem4.2-pos}, for any $\varphi \in \mathcal{X}^{1,2}_0(\Omega)$,
\begin{align*}
\langle J'(u_k-u), \varphi \rangle
&= \int_{\mathbb{R}^n} \nabla (u_k-u) \cdot \nabla \varphi \, dx  \\
&\quad - \int_{\Omega} \int_{\Omega}
\frac{|(u_k-u)_+(x)|^{2^*_\mu - 2}
|(u_k-u)_+(y)|^{2^*_\mu}
(u_k-u)_+(x)\,\varphi(x)}
{|x-y|^\mu} \, dx \, dy \\[1mm]
&= \int_{\mathbb{R}^n} \nabla u_k \cdot \nabla \varphi \, dx
- \int_{\mathbb{R}^n} \nabla u \cdot \nabla \varphi \, dx \\[1mm]
&\quad - \int_{\Omega} \int_{\Omega}
\frac{|(u_k)_+(x)|^{2^*_\mu - 2}
|(u_k)_+(y)|^{2^*_\mu}
(u_k)_+(x)\,\varphi(x)}
{|x-y|^\mu} \, dx \, dy \\[1mm]
&\quad + \int_{\Omega} \int_{\Omega}
\frac{|u_+(x)|^{2^*_\mu - 2}
|u_+(y)|^{2^*_\mu}
u_+(x)\,\varphi(x)}
{|x-y|^\mu} \, dx \, dy
- \mathcal{N}_k(\varphi)
+ o(1) \\[1mm]
&= \langle I'(u_k), \varphi \rangle
- \langle I'(u), \varphi \rangle
+ o(1) = o(1).
\end{align*}
since $\{u_k\}$ is a P.S. sequence and $u$ is a weak solution.
% \[
% \langle J'(v_k),\, \varphi \rangle = 
% \left( \langle I'(u_k),\, \varphi \rangle - \langle I'(u),\, \varphi \rangle \right) + o(1) \to 0.
% \]
Hence, $J'(v_k) \to 0$ in $(\mathcal{X}^{1,2}_0(\Omega))'$, due to the  claim proved above. Moreover, using Lemma~\ref{compact_embedding}, we can easily see that $\{v_k\}$ is indeed a $(PS)$ sequence for $J$ at the level $c-I(u)$ on $H^1_0(\Omega)$.
\end{proof}
\QED

%%%%%%%%%%%%%%%%%%%%%%%%%%%
We now discuss the concentration profile of the sequence $\{u_k\}$.
\begin{Lemma}\label{lem4.4}
Let $\{y_k\} \subset \Omega$ and $\{\lambda_k\} \subset (0,\infty)$ satisfy
\(
\frac{1}{\lambda_k}\,\operatorname{dist}(y_k,\partial\Omega)\to\infty
\), as $k\to \infty.$
Assume a sequence $\{u_k\}\in \mathcal{X}^{1,2}_0(\Omega)$ and its rescaled version
\[
f_k(x)=\lambda_k^{\frac{n-2}{2}}u_k(\lambda_k x+y_k)
\]
is such that for some $f\in D^{1,2}(\mathbb{R}^n)$,
$f_k\rightharpoonup f$ weakly in $D^{1,2}(\mathbb{R}^n)$,
$f_k\to f$ a.e. in $\mathbb{R}^n$,
$J(u_k)\to c$ in $\mathbb R$, and $J'(u_k)\to0$ in $(\mathcal{X}^{1,2}_0(\Omega) )'$ as $k\to \infty$. Then $J'(f)=0$. Moreover, defining
\[
z_k(x)=u_k(x)-\lambda_k^{-\frac{n-2}{2}}f\!\left(\frac{x-y_k}{\lambda_k}\right),
\]
one has
\begin{enumerate}
\item $J(z_k)\to c-J(f)$ as $k\to \infty$,
    \item $J'(z_k)\to0$ in $(\mathcal{X}^{1,2}_0(\Omega) )'$ as $k\to \infty$.
\end{enumerate}
\end{Lemma}

\begin{proof} (1)
For any $\varphi \in C_c^\infty(\mathbb{R}^n)$, define
\[
\varphi_k(x) := \lambda_k^{\frac{2-n}{2}}\,\varphi\!\left(\frac{x-y_k}{\lambda_k}\right).
\]
If $\varphi \in C_c^\infty(B_r)$, then for all sufficiently large $k$ we have $\varphi_k \in C_c^\infty(\Omega)$. 
It implies 
\begin{align}\label{scalingp}
	\ld J^\prime(f_k), \varphi \rd =& \int_{\mathbb{R}^n} \na f_k\cdot\na \varphi \,dx 
	-\I{\mathbb{R}^n}\I{\mathbb{R}^n}\frac{|(f_k)_+(y)|^{2_\mu^*}|(f_k)_+(x)|^{2_\mu^*-1}\varphi(x)}{|x-y|^\mu}dxdy\\\nonumber
    =& \int_{\mathbb{R}^n} \na u_k\cdot\na \varphi_k \,dx 
	-\I{\mathbb{R}^n}\I{\mathbb{R}^n}\frac{|(u_k)_+(y)|^{2_\mu^*}|(u_k)_+(x)|^{2_\mu^*-1}\varphi_k(x)}{|x-y|^\mu}dxdy\\\nonumber
    =& \ld J^\prime(u_k), \varphi_k \rd.
\end{align} 
It follows that
\begin{equation}\label{new1}
\big|\, \langle J'(f_k),\, \varphi \rangle\,\big|
= \left|\, \langle J'(u_k),\, \varphi_k \rangle\,\right|
\leq \|J'(u_k)\|_{(\mathcal{X}^{1,2}_0(\Omega) )'}\, \|\nabla \varphi_k\|_{L^2(\Omega)} = \|J'(u_k)\|_{(\mathcal{X}^{1,2}_0(\Omega ))'}\, \|\nabla \varphi\|_{L^2(\Omega)} \to 0
    \end{equation}
    as $k\to \infty$, since $J'(u_k) \to 0$ as $k\to \infty$ in $(\mathcal{X}^{1,2}_0(\Omega) )'$.
Hence,$J'(f_k) \to 0$ as $k \to \infty$ in $(\mathcal{X}^{1,2}_0(B_r))'$ for each $r>0$ (since $r$ is arbitrary).
Now, we aim to show that $J'(f) = 0$.
For any $\varphi \in C_c^\infty(\mathbb{R}^n)$, we easily get that $\varphi \in C_c^\infty(B_r)$ for some $r$. Using
\(
\frac{1}{\lambda_k} \operatorname{dist}(y_k, \partial\Omega) \to \infty,
\)
together with $J'(f_k) \to 0$ in $(\mathcal{X}^{1,2}_0(B_r))'$ as $k\to \infty$ and following the same reasoning as in the proof of Lemma \ref{lem4.2}, due to \eqref{new1} we obtain 
\[
\langle J'(f_k) - J'(f),\, \varphi \rangle \to 0,
\]
which establishes that $J'(f) = 0$ via density arguments.
Now, using Lemma \ref{1lem4.2-pos}, and assuming that
\[
f_k \rightharpoonup f \quad \text{in } D^{1,2}(\mathbb{R}^n),
\qquad
f_k \to f \quad \text{a.e. in } \mathbb{R}^n,
\]
we obtain the following decomposition
\begin{align*}
\int_{\mathbb{R}^n}\int_{\mathbb{R}^n}
\frac{|(f_k)_+(x)|^{2^*_\mu}|(f_k)_+(y)|^{2^*_\mu}}{|x-y|^\mu}
\,dx\,dy
&=
\int_{\mathbb{R}^n}\int_{\mathbb{R}^n}
\frac{|f_+(x)|^{2^*_\mu}|f_+(y)|^{2^*_\mu}}{|x-y|^\mu}
\,dx\,dy \\
&\quad +
\int_{\mathbb{R}^n}\int_{\mathbb{R}^n}
\frac{|(f_k-f)_+(x)|^{2^*_\mu}|(f_k-f)_+(y)|^{2^*_\mu}}{|x-y|^\mu}
\,dx\,dy + o(1),
\end{align*}
where the term $o(1)$ denotes tends to $0$ as $k \to \infty$.
Again, by the classical Brezis-Lieb lemma \cite{gao2016brezis}, we have the following
\begin{equation}\label{BL-grad}
\int_{\mathbb{R}^n} |\nabla f_k|^2 \, dx
=
\int_{\mathbb{R}^n} |\nabla (f_k - f)|^2 \, dx
+
\int_{\mathbb{R}^n} |\nabla f|^2 \, dx
+ o(1).
\end{equation}
Therefore, using the definition of the functional $J$, we compute
\begin{align*}
J(f_k)
&=
\frac12 \int_{\mathbb{R}^n} |\nabla f_k|^2 \, dx
-
\frac{1}{2\cdot 2^*_\mu}
\int_{\mathbb{R}^n}\int_{\mathbb{R}^n}
\frac{|(f_k)_+(x)|^{2^*_\mu}|(f_k)_+(y)|^{2^*_\mu}}{|x-y|^\mu}
\,dx\,dy \\
&=
\frac12
\Bigl(
\int_{\mathbb{R}^n} |\nabla (f_k-f)|^2 \, dx
+
\int_{\mathbb{R}^n} |\nabla f|^2 \, dx
\Bigr) \\
&\quad -
\frac{1}{2\cdot 2^*_\mu}
\Biggl[
\int_{\mathbb{R}^n}\int_{\mathbb{R}^n}
\frac{|(f_k-f)_+(x)|^{2^*_\mu}|(f_k-f)_+(y)|^{2^*_\mu}}{|x-y|^\mu}
\,dx\,dy \\
&\qquad\qquad +
\int_{\mathbb{R}^n}\int_{\mathbb{R}^n}
\frac{|f_+(x)|^{2^*_\mu}|f_+(y)|^{2^*_\mu}}{|x-y|^\mu}
\,dx\,dy
\Biggr]
+ o(1).
\end{align*}
Consequently, we obtain the energy decomposition
\[
J(f_k) = J(f_k - f) + J(f) + o(1).
\]
Letting $z_k := f_k - f$. Since $J(u_k) = J(f_k)$ and $J(u_k) \to c$, it follows that
\[
J(z_k)
=
J(f_k) - J(f) + o(1)
\longrightarrow c - J(f).
\]
(2) 
Recall that
\[
z_k=u_k-w_k,
\qquad
w_k:=\lambda_k^{\frac{2-n}{2}}
f\!\left(\frac{\cdot-y_k}{\lambda_k}\right).
\]
For any $\varphi\in \mathcal{X}^{1,2}_0(\Omega) $, by the definition of $J'$ we write
\begin{align*}
\langle J'(z_k),\varphi\rangle
&=
\int_{\mathbb R^n}\nabla(u_k-w_k)\cdot\nabla\varphi\,dx \\
&\quad-
\int_{\mathbb R^n}
\Big(|x|^{-\mu}*|(u_k-w_k)_+|^{2^*_\mu}\Big)
|(u_k-w_k)_+|^{2^*_\mu-2}(u_k-w_k)_+\,\varphi\,dx .
\end{align*}
By definition of $z_k$, we obtain
\begin{align*}
\int_{\mathbb R^n}\nabla z_k\cdot\nabla\varphi\,dx
&=
\int_{\mathbb R^n}\nabla u_k\cdot\nabla\varphi\,dx
-
\int_{\mathbb R^n}\nabla w_k\cdot\nabla\varphi\,dx .
\end{align*}
Moreover, by Lemma \ref{1lem4.2-pos} (Brezis-Lieb type decomposition for the Choquard
nonlinearity), we have 
\begin{align*}
&\Big(|x|^{-\mu}*|(u_k-w_k)_+|^{2^*_\mu}\Big)
|(u_k-w_k)_+|^{2^*_\mu-2}(u_k-w_k)_+ \\
&\qquad=
\Big(|x|^{-\mu}*|(u_k)_+|^{2^*_\mu}\Big)
|(u_k)_+|^{2^*_\mu-2}u_k \\
&\qquad\quad-
\Big(|x|^{-\mu}*|(w_k)_+|^{2^*_\mu}\Big)
|(w_k)_+|^{2^*_\mu-2}w_k
+o(1).
\end{align*}
Testing the above identity with $\varphi$ and combining the previous
estimates, we deduce
\begin{align*}
\langle J'(z_k),\varphi\rangle
&=
\Big[
\int_{\mathbb R^n}\nabla u_k\cdot\nabla\varphi\,dx
-
\int_{\mathbb R^n}
\big(|x|^{-\mu}*|(u_k)_+|^{2^*_\mu}\big)
|(u_k)_+|^{2^*_\mu-2}(u_k)_+\,\varphi\,dx
\Big] \\
&\quad-
\Big[
\int_{\mathbb R^n}\nabla w_k\cdot\nabla\varphi\,dx
-
\int_{\mathbb R^n}
\big(|x|^{-\mu}*|(w_k)_+|^{2^*_\mu}\big)
|(w_k)_+|^{2^*_\mu-2}(w_k)_+\,\varphi\,dx
\Big]
+o(1),
\end{align*}
that is,
\[
\langle J'(z_k),\varphi\rangle
=
\langle J'(u_k),\varphi\rangle
-
\langle J'(w_k),\varphi\rangle
+o(1),
\]
also we can write
\[
\langle J'(z_k),\varphi\rangle = 
\langle J'(u_k) - J'\bigl(\lambda_k^{\frac{2-n}{2}} f\bigl(\tfrac{\cdot-y_k}{\lambda_k}\bigr)\bigr),\,\varphi\rangle + o(1).
\]
Since the rescaling $\lambda_k^{\frac{2-n}{2}} f\bigl(\tfrac{\cdot-y_k}{\lambda_k}\bigr)$ preserves the form $J'$ as in \eqref{scalingp} and using $J'(f)=0$ on $\mathbb{R}^n$, it follows that 
\[
J'\bigl(\lambda_k^{\frac{2-n}{2}} f\bigl(\tfrac{\cdot-y_k}{\lambda_k}\bigr)\bigr)=0
\]
in $(\mathcal{X}^{1,2}_0(\Omega))'$ for any $k$. 
Hence,
\[
\langle J'(z_k),\varphi\rangle = \langle J'(u_k),\varphi\rangle + o(1) = o(1),
\]
which implies $J'(z_k) \to 0$ in $(\mathcal{X}^{1,2}_0(\Omega))'$, since $ J'(u_k) \to 0$, by assumption.
\end{proof}\QED

%%%%%%%%%%%%%%%%%%%%%%%%%%%%
Before establishing the global compactness result (Theorem\ref{PS_decomposition}) for the Choquard equation, we recall the definition of \emph{Morrey spaces} $\mathcal{L}^{(r,\gamma)}$, introduced by Morrey, giving a refinement of the usual Lebesgue spaces. 

\begin{Definition}\textbf{[Morrey Spaces]}\label{Ms}
A measurable function $u : \mathbb{R}^n \to \mathbb{R}$ belongs to the Morrey space $\mathcal{L}^{(r,\gamma)}(\mathbb{R}^n)$, with $r \in [1,\infty)$ and $\gamma \in [0,n]$, if and only if
\begin{equation}\label{eq:morrey-norm}
\|u\|_{\mathcal{L}^{r,\gamma}(\mathbb{R}^n)}^r
:= \sup_{R>0,\, x\in\mathbb{R}^n} R^{{\gamma}-n} \int_{B_R(x)} |u(y)|^r \,dy < \infty.
\end{equation}
Moreover,  we have the continuous embedding 
\[
L^{2^*}(\mathbb{R}^n) \hookrightarrow L^{(2,\,n-2)}(\mathbb{R}^n).
\]
\end{Definition}
%%%%%%%%%%%%%%%%%%%%%%%%%%%%%%%%%

Now, with the help of Lemmas \ref{lem4.2} and \ref{lem4.4}, we can give the proof of our first main result, Theorem \ref{PS_decomposition}.

\textbf{Proof of Theorem \ref{PS_decomposition}:}
We divide the proof into several steps.

\textbf{Step 1.}  
We first show that the Palais-Smale sequence $\{u_k\}$ is bounded. 
Let $\{u_k\}\subset \mathcal{X}^{1,2}_0(\Omega)$ be such that $I(u_k)\to c$ and $I'(u_k)\to 0$. 
Then, by the definition of $I$,
\[
I(u_k) - \frac{1}{2\cdot 2^*_{\mu}} \langle I'(u_k), u_k\rangle 
= \left(\frac{1}{2} - \frac{1}{2\cdot 2^*_{\mu}}\right)\eta(u_k)^2.
\]
Since $\{I(u_k)\}$ is bounded and $\langle I'(u_k), u_k\rangle \to 0$, it follows that
\[
\left(\frac{1}{2} - \frac{1}{2\cdot 2^*_{\mu}}\right)\eta(u_k)^2 \leq C + o(1),
\]
for some constant $C>0$. 
Since $2^*_{\mu} > 1$, $\left(\tfrac{1}{2} - \tfrac{1}{2\cdot 2^*_{\mu}}\right)>0$. 
 Hence, $\{u_k\}$ is bounded in $\mathcal{X}^{1,2}_0(\Omega) $. It implies that there exists a $v_0 \in \mathcal{X}^{1,2}_0(\Omega) $ such that $u_k \rightharpoonup v_0$ weakly in $\mathcal{X}^{1,2}_0(\Omega) $ and $u_k \to v_0$ a.e. in $\Omega$. By Lemma \ref{lem4.2}, $I'(v_0) = 0$. Set $u_k^1 := u_k - v_0$. Then
\begin{equation}\label{eq411}
\eta(u_k^1)^2 = \eta(u_k)^2 - \eta(v_0)^2 + o(1), \qquad J(u_k^1) \to c - I(v_0), \qquad J'(u_k^1) \to 0 \text{ in } (\mathcal{X}^{1,2}_0(\Omega) )'.
    \end{equation}
Moreover, since $\{u_k\}$ is bounded, $v_0 \in \mathcal{X}^{1,2}_0(\Omega) $ and
\[
\eta(u_k^1)^2 = \eta(u_k)^2 - \eta(v_0)^2 + o(1),
\]
one gets that the sequence $\{u_k^1\}$ is also bounded in $\mathcal{X}^{1,2}_0(\Omega) $. Therefore, there exists a constant $C_1 > 0$ such that
\begin{equation}\label{eno}
\eta(u_k^1) < C_1 \quad \text{for all } k \in \mathbb{N}.
    \end{equation}
\vspace{0.1em}
\noindent\textbf{Step 2.} In this step 2, we aim to detect the first concentration point and extract a non-trivial bubble profile $v_1$ when the sequence $u^1_k$ fails to converge strongly to zero.

Now suppose
\[
\int_{\Omega} \int_{\Omega} \frac{|(u_k^1)_+(x)|^{2^*_{\mu}} |(u_k^1)_+(y)|^{2^*_{\mu}}}{|x - y|^\mu} \, dx dy \to 0.
\]
Then, using the energy expansion and the fact that $J'(u_k^1) \to 0$, we obtain
  $u_k^1 \to 0$ in $\mathcal{X}^{1,2}_0(\Omega) $ and
 we are done. On the other hand, if
\[
\int_{\Omega} \int_{\Omega}  \frac{ |(u_k^1)_+(x)|^{2^*_{\mu}} |(u_k^1)_+(y)|^{2^*_{\mu}}}{ |x - y|^\mu } dx dy \not\to 0, ~~\text{as}~~k\to \infty,
\]
then we may assume up to a subsequence that
\begin{equation}\label{eqN}
\int_{\Omega} \int_{\Omega}  \frac{ |(u_k^1)_+(x)|^{2^*_{\mu}} |(u_k^1)_+ (y)|^{2^*_{\mu}}}{ |x - y|^\mu } dx dy > \delta
    \end{equation}
for some $\delta > 0$.
This, together with the Hardy–Littlewood–Sobolev inequality, gives $\|u_k^1\|_{L^{2^*}(\Omega)} > \delta_1$ for all $k$ and for an appropriate constant $\delta_1>0$. Taking into account that $u_k^1$ is a bounded sequence in $L^{2^*}(\mathbb{R}^n)$ and $L^{2^*}(\mathbb{R}^n) \hookrightarrow L^{(2,n-2)}(\mathbb{R}^n)$, and the Gagliardo-Nirenberg type inequality [\cite{palatucci2014improved}, Theorem 1.2], we obtain
\[
b < \|u_k^1\|_{L^{(2,n-2)}(\mathbb{R}^n)} < a \quad \text{for all } k.
\]
Thus, we can observe that there exists a positive constant $C$ small enough such that for all $k$, we get
\begin{equation}\label{eq410}
C < \|u_k^1\|_{L^{(2,n-2)}(\mathbb{R}^n)} < C^{-1}.
\end{equation}
By the definition of Morrey spaces, for $r = 2$ and $\gamma = n - 2$, we have
\[
\|u_k^1\|_{\mathcal{L}^{(2,n-2)}(\mathbb{R}^n)}^2 = \sup_{x \in \mathbb{R}^n,\, R > 0} R^{-2} \int_{B_R(x)} |u_k^1(y)|^2 \, dy.
\]
Since $\|u_k^1\|_{\mathcal{L}^{(2,n-2)}(\mathbb{R}^n)}$ is uniformly bounded below as in \eqref{eq410}, it follows that for each $k \in \mathbb{N}$ there exist points $y_k^1 \in \mathbb{R}^n$ and radii $\lambda_k^1 > 0$ such that
\begin{equation}\label{eq:concentration}
0 < \widehat{C} < \|u_k^1\|_{\mathcal{L}^{(2,n-2)}(\mathbb{R}^n)}^2 - \frac{C^2}{2k} < (\lambda_k^1)^{-2} \int_{B_{\lambda_k^1}(y_k^1)} |u_k^1(y)|^2 \, dy,
\end{equation}
for some constant $\widehat{C} > 0$ independent of $k$.
 Now, define the rescaled function
\[
f_k^1(x) := (\lambda_k^1)^{\frac{n - 2}{2}} u_k^1(\lambda_k^1 x + y_k^1).
\]
Since $\|\nabla f^1_k\|_{L^2(\mathbb{R}^n)}\leq \eta (u^1_k)$ and using \eqref{eno}, we have $\|\nabla f^1_k\|_{L^2(\mathbb{R}^n)}< C$, independently of $k\in\mathbb{N}$. Up to a subsequence, one has that $f^1_k \rightharpoonup v_1$ weakly in $D^{1,2}(\mathbb{R}^n)$ and$f^1_k \to v_1$ a.e. in $\mathbb{R}^n.$ Furthermore, changing variables in the integral over $B_1(0)$, we compute
\[
\int_{B_1(0)} |f_k^1(x)|^2 dx 
= (\lambda_k^1)^{n - 2} \int_{B_1(0)} |u_k^1(\lambda_k^1 x + y_k^1)|^2 dx.
\]
Now let $y = \lambda_k^1 x + y_k^1$ so that $dx = (\lambda_k^1)^{-n} dy$ and $x \in B_1(0) \iff y \in B_{\lambda_k^1}(y_k^1)$. Thus,
\[
\int_{B_1(0)} |f_k^1(x)|^2 dx 
= (\lambda_k^1)^{-2} \int_{B_{\lambda_k^1}(y_k^1)} |u_k^1(y)|^2 dy > \widehat{C}.
\]
Since $D^{1,2}(\mathbb{R}^n) \hookrightarrow L^2_{\mathrm{loc}}(\mathbb{R}^n)$ is compact, we have $\int_{B_1(0)} |v_1|^2 dx > \widehat{C} > 0$. It implies that $v_1 \neq 0$.
\vspace{2ex}

\noindent\textbf{Step 3.} We claim that \( \lambda_k^1 \to 0 \) and \( y_k^1 \to y_0 \in \overline{\Omega} \). Suppose, by contradiction, that \( \lambda_k^1 \to \infty \) as \( k \to \infty \). Since \( \{u_k^1\} \subset \mathcal{X}^{1,2}_0(\Omega) \) is bounded, it is also bounded in \( L^2(\Omega) \). 
Define the rescaled domains
\[
\Omega_k := \frac{\Omega - y_k^1}{\lambda_k^1},
\]
so that the change of variable \( y = \lambda_k^1 x + y_k^1 \) yields
\[
\int_{\Omega_k} |f_k^1(x)|^2\, dx 
= {(\lambda_k^1)^{-2}} \int_{\Omega} |u_k^1(y)|^2\, dy 
\leq {C}{(\lambda_k^1)^{-2}} \to 0 \quad \text{as } k \to \infty.
\]
On the other hand, by Fatou's lemma, we get,
\[
0 = \liminf_{k \to \infty} \int_{\Omega_k} |f_k^1(x)|^2\, dx 
\geq \int_{ \Omega_k} |v_1(x)|^2\, dx
\]
 which implies $v_1\equiv 0$ in $\Omega_k$. Since \( v_1 \not\equiv 0 \) by Step 2, the above yields a contradiction. Hence, the sequence \( \{\lambda_k^1\} \) must be bounded in \( \mathbb{R}^+ \), and up to a subsequence, there exists \( \lambda_*^1 \in [0, \infty) \) such that \( \lambda_k^1 \to \lambda_*^1 \) as $k \to \infty$.
 We now claim that \( \{y_k^1\} \) is bounded in \( \overline{\Omega} \). Suppose, by contradiction, that \( |y_k^1| \to \infty \) as \( k \to \infty \). Then, for any fixed \( x \in \mathbb{R}^n \), the point \( \lambda_k^1 x + y_k^1 \notin \overline{\Omega} \) for all large \( k \).
 %because \( \lambda_k^1 x \) remains bounded while \( y_k^1 \) escapes to $\infty$.
Since $u_k \in \mathcal{X}^{1,2}_0(\Omega) $, it follows that for large enough $k$, $u_k^1(\lambda_k^1 x + y_k^1) = 0$ for all $x \in \Omega$, which contradicts \eqref{eqN}.
Therefore, the sequence $\{y_k^1\}$ is bounded, and up to a subsequence, we have $y_k^1 \to y_0^1 \in \mathbb{R}^n$. Now suppose, for contradiction, that $\lambda_k^1 \to \lambda_*^1 > 0$. Then the rescaled domains
\[
\Omega_k := \frac{\Omega - y_k^1}{\lambda_k^1} \to \frac{\Omega - y_0^1}{\lambda_*^1} =: \Omega_0 \subsetneq \mathbb{R}^n.
\]
Now, using the fact that $u_k^1 \rightharpoonup 0$ weakly in $\mathcal{X}^{1,2}_0(\Omega) $, it follows that $f_k^1 \rightharpoonup 0$ weakly in $D^{1,2}(\mathbb{R}^n)$. However, this contradicts the fact that $f_k^1 \to v_1 \not\equiv 0$ in $D^{1,2}(\mathbb{R}^n)$ (as shown in Step 2). Therefore, the assumption must be false, and we conclude that
\[
\lambda_k^1 \to 0 \quad \text{as } k \to \infty.
\]
We now argue by contradiction and assume that
\begin{equation}\label{4.11}
y_0^1 \notin \overline{\Omega}.
\end{equation}
Note that for any fixed $x \in \Omega$, we have $\lambda_k^1 x + y_k^1 \to y_0^1$ as $k \to \infty$. In view of \eqref{4.11}, it follows that for all $x \in \Omega$, there exists $k_0 \in \mathbb{N}$ such that for all $k \geq k_0$, we have
\(
\lambda_k^1 x + y_k^1 \notin \overline{\Omega}.
\)
Since $u_k^1 \in \mathcal{X}^{1,2}_0(\Omega) $, this implies that $u_k^1(\lambda_k^1 x + y_k^1) = 0$ for all $x \in \Omega$ and $k$ large enough. Consequently, $f_k^1(x) = 0$ for all $x \in \Omega$ and large $k$, which contradicts the fact that $f_k^1 \to v_1 \not\equiv 0$. 

Therefore, our assumption \eqref{4.11} must be false, and we conclude that
\(
y_0^1 \in \overline{\Omega}.
\)
This completes the proof of the claim and Step 3.
\vspace{1em}

\noindent\textbf{Step 4.} Assume that
\[
\lim_{k \to \infty} \frac{1}{\lambda_k^1} \operatorname{dist}(y_k^1, \partial\Omega) \to c < \infty.
\]
Then, utilizing a Liouville-type result \cite[theorem 2.8]{goel2020coron}, we obtain $v_1 \equiv 0$, which is not possible. Therefore,
\[
\frac{1}{\lambda_k^1} \operatorname{dist}(y_k^1, \partial \Omega) \to \infty \quad \text{as } k \to \infty.
\]
Thus by \eqref{eq411} and Lemma \ref{lem4.4}, we have $J'(v_1) = 0$ and the sequence
\[
u_k^2(x) = u_k^1(x) - (\lambda_k^1)^{\frac{2-n}{2}} v_1\left( \frac{x - y_k^1}{\lambda_k^1} \right)
\]
satisfies
\[
J(u_k^2) \to c - J(v_0) - J(v_1), \quad J'(u_k^2) \to 0 \text{ in } (\mathcal{X}^{1,2}_0(\Omega) )'.
\]
By {\cite[Proposition 2.6]{goel2020coron}, we have \( J(v_1) \geq \beta > 0 \)}. Repeating this procedure, we construct sequences \( \{v_i\} \subset D^{1,2}(\mathbb{R}^n) \), \( \{y_k^i\} \subset \Omega \), \( \{\lambda_k^i\} \subset \mathbb{R}_+ \), and corresponding remainders \( u_k^{i+1} \) such that after \( j \) steps,
\[
J(u_k^{j+1}) < I(u_k) - I(v_0) - \sum_{i=1}^j J(v_i) \leq I(u_k) - I(v_0) - kc.
\]
Since the right-hand side becomes negative for large \( k \), the process must terminate after finitely many steps. As a result, we obtain \( k \in \mathbb{N} \), sequences \( \{y_k^i\} \subset \Omega \), \( \{\lambda_k^i\} \subset \mathbb{R}_+ \), and associated nontrivial solutions \( \{v_i\} \) of \eqref{eq4.7} satisfying  condition~\eqref{eq4.8}.
\QED

We now prove that the functional $I$ satisfies the Palais-Smale condition at levels strictly below the compactness threshold.

\begin{Lemma}\label{beta}
The functional $I$ satisfies the Palais-Smale condition for any $c \in (\beta, 2\beta)$, where
\[
\beta = \frac{1}{2} \left( \frac{n - \mu + 2}{2n - \mu} \right) S_{H,L,C}^{\frac{2n - \mu}{n - \mu + 2}}.
\]
\end{Lemma}
\begin{proof}
Let \( c \in (\beta, 2\beta) \), and suppose there \( \{u_k\} \subset \mathcal{X}^{1,2}_0(\Omega) \) is a sequence such that
\[
I(u_k) \to c \quad \text{and} \quad I'(u_k) \to 0 \quad \text{in } (\mathcal{X}^{1,2}_0(\Omega) )'.
\]
By the global compactness Theorem \ref{PS_decomposition}, up to a subsequence, there exists a function \( v_0 \in \mathcal{X}^{1,2}_0(\Omega) \) which is a solution of \eqref{eq4.6} satisfies \( I(v_0) \geq \beta \)
and an integer \( k \in \mathbb{N} \cup \{0\} \), along with nontrivial solutions \( \{v_i\}_{i=1}^k \subset D^{1,2}(\mathbb{R}^n) \) of the problem
\eqref{eq4.7}
and  sequences \( \{y_k^i\} \subset \mathbb{R}^n \) and \( \{\lambda_k^i\} \subset \mathbb{R}_+ \), such that
\[
\frac{1}{\lambda_k^i} \operatorname{dist}(y_k^i, \partial \Omega) \to \infty \quad \text{as } k \to \infty,
\]
and
\[
\left\| u_k - v_0 - \sum_{i=1}^k (\lambda_k^i)^{\frac{2 - n}{2}} v_i\left( \frac{x - y_k^i}{\lambda_k^i} \right) \right\|_{D^{1,2}(\mathbb{R}^n)} \to 0 \quad \text{as } k \to \infty,
\]
together with the energy decomposition
\[
 I(u_k)\to  I(v_0) + \sum_{i=1}^k J(v_i) = c, ~~~\text{as}~~k\to \infty.
\]
 Therefore, the energy decomposition implies
\[
c = I(v_0) + \sum_{i=1}^k J(v_i) \geq k \beta.
\]
Since \( c < 2\beta \), it follows that \( k \leq 1 \).

If \( k = 0 \), then the decomposition reduces to \( u_k \to v_0 \) strongly in \( \mathcal{X}^{1,2}_0(\Omega) \), and the Palais-Smale condition is satisfied. Now suppose \( k = 1 \), so that
\[
c = I(v_0) + J(v_1).
\]
If \( v_0 \not\equiv 0 \), then by Lemma~\ref{1lem2.8}, we have \( I(v_0) \geq \beta \), and hence \( J(v_1) = c - I(v_0) < \beta \). However, this contradicts the fact that \( J(v_1) \geq \beta \). Therefore, this case is not possible. If $v_0 \equiv 0$, then by Theorem \ref{PS_decomposition}, $J(v_1) = c$ and $v_1$ is a nonnegative solution of \eqref{eq4.7}. In the view of  [\cite{goel2020coron}, Corollary 3.5], \( v_1 \) is radially symmetric, decreasing, and of the form
\[
v_1(x) = \left( \frac{a}{b + |x - x_0|^2} \right)^{\frac{n - 2}{2}},
\]
for some constants \( a, b > 0 \) and \( x_0 \in \mathbb{R}^n \). Then by [\cite{goel2020coron},{Corollary 3.5}], this function achieves the best constant \( S_{H,L,C} \), and thus \( J(v_1) = \beta \), contradicting the fact that \( c > \beta \). Therefore, this case also cannot occur.
We conclude that the only possibility is \( k = 0 \), in which case \( u_k \to v_0 \) strongly in \( \mathcal{X}^{1,2}_0(\Omega) \). This proves that the functional \( I \) satisfies the Palais-Smale condition for any \( c \in (\beta, 2\beta) \).
\end{proof}
\QED

\section{Positive solutions on annular-shaped domains}

To establish Theorem~\ref{mainthm}, we first derive some auxiliary results that play a crucial role in the subsequent construction. Our approach relies on introducing a two-parameter family of Aubin-Talenti type bubbles whose energy remains strictly between the first and second levels.

We assume that $\Omega$ satisfies conditions \eqref{eq1.6} and \eqref{eq1.7}. For simplicity of exposition, we assume \( R_1, R_2 \) to be the radii of the annulus as described in Theorem~\ref{mainthm}. Without loss of generality, we assume \( R_1 = \frac{1}{4R} <1< 4R= R_2 \), and \( R >>1 \) is a parameter to be chosen sufficiently large.

We consider the following family of functions
\begin{equation}\label{eq5.1}
u_t^\sigma(x) := \mathcal{U}[t\sigma, (1-t)\sigma](x)=  
\left( \frac{1-t}{(1-t)^2 + |x - t\sigma|^2} \right)^{\frac{n-2}{2}}, \quad x \in \mathbb{R}^n,
    \end{equation}
where \( \sigma \in \Sigma := \{ x \in \mathbb{R}^n : |x| = 1 \} \), and \( t \in [0, 1) \). These functions are motivated from the fact that $S_{H,L,C}$ is attained at each of $u_{t}^{\sigma}$ and $S_{H,L,C} = S_{H,L,M}$ as discussed in Theorem \ref{thm:best_constant}. These functions belong to \( D^{1,2}(\mathbb{R}^{n}) \) and are known to concentrate at the point \( \sigma \) as \( t \to 1 \). Moreover, as \( t \to 0 \), we have
\begin{equation}\label{eq5.2}
u_t^\sigma \to u_0(x) = \left( \frac{1}{1 + |x|^2} \right)^{\frac{n-2}{2}}.
    \end{equation}
 To localize these function within the domain $\Omega$, consider a radially symmetric cut-off function
\( \nu \in C_0^\infty(\mathbb{R}^{n}) \) which satisfying \( 0 \leq \nu \leq 1 \) in \( \Omega \), and
\[
\nu(x) = 
\begin{cases}
1, & \text{if } \frac{1}{2} < |x| < 2, \\
0, & \text{if } |x| > 4 \text{ or } |x| < \frac{1}{4}.
\end{cases}
\]
For each $R\geq 1$, we define the following rescaled cut-off function \( \nu \) as
\begin{equation}\label{nu}
\nu_R(x) = 
\begin{cases}
\nu(Rx), & 0 \leq  |x| < \frac{1}{R}, \\
1, & \frac{1}{R} \leq |x| <R, \\
\nu\left(\frac{x}{R}\right), & |x| \geq  R.
\end{cases}
    \end{equation}
With this, we define truncated functions
\begin{equation}\label{gt}
g_t^\sigma(x) := u_t^\sigma(x)\, \nu_R(x) , \quad \text{and} \quad g_0(x) := u_0(x)\, \nu_R(x) \in \mathcal{X}^{1,2}_0(\Omega).
    \end{equation}
    %%%%%%%%%%%%%%%%%%%%%%%%%%%%%%%%%%%%%%%%%%%%%

     \begin{Lemma}\label{lem:bubbling-profile}
Let \( \sigma \in \Sigma \), and for \( t \in [0,1) \). 
Then for each \( u_t^\sigma \in D^{1,2}(\mathbb{R}^{n}) \) satisfies the following
\begin{itemize}
    \item[(i)] \( \|u_t^\sigma\|_{D^{1,2}(\mathbb{R}^{n})} = \|u_0\|_{D^{1,2}(\mathbb{R}^{n})} \),
    \item[(ii)] \( \|u_t^\sigma\|_0 = \|u_0\|_0 \), where 
    \begin{equation}\label{no0}
   \|u\|_0 := \left( \int_{\mathbb{R}^n} \int_{\mathbb{R}^n} \frac{|u_+(x)|^{2^*_{\mu}} |u_+(y)|^{2^*_{\mu}}}{|x - y|^\mu} \, dx \, dy \right)^{\frac{1}{2 \cdot 2^*_{\mu}}}.
         \end{equation}
\end{itemize}
\end{Lemma}

\begin{proof}
Let \( \lambda_t := \frac{1}{1 - t} \), so that
\[
u_t^\sigma(x) = \lambda_t^{\frac{n - 2}{2}} u_0(\lambda_t(x - t\sigma)),
\]
where $u_t^\sigma(x)$ $u_0(x)$ is defined by \eqref{eq5.1} and \eqref{eq5.2} respectively.

(i) \textit{(Invariance of the norm under translation and dilation)}.  
Since the norm is invariant under the scaling, the norm \( \|u_t^\sigma\|_{D^{1,2}(\mathbb{R}^{n})} \) remains constant with respect to \( t \). Thus, \( \|u_t^\sigma\|_{D^{1,2}(\mathbb{R}^{n})} = \|u_0\|_{D^{1,2}(\mathbb{R}^{n})} \).

(ii) \textit{(Scaling invariance of the Choquard-type norm)}.  
The integral \( \|u\|_0 \) is invariant under the scaling.
%and translation used to define \( u_t^\sigma \), due to the critical exponent and homogeneity of the kernel \( |x - y|^{-\mu} \). A change of variables confirms that the norm remains unchanged, 
So \( \|(u_t^\sigma)_+\|_0 = \|(u_0)_+\|_0 \).
\end{proof}\QED

%%%%%%%%%%%%%%%%%%%%%%%%%%%%%%%%

The next result follows directly from~\cite[Lemma~4.1]{chakraborty2025global}.

\begin{Lemma}\label{Lem:seminormConvergence}
With the above notations, we have
\[
\left\| g_{t}^\sigma - u_{t}^\sigma \right\|_{\X(\mathbb{R}^{n})} \longrightarrow 0 
\quad \text{as } R \to \infty,
\]
uniformly with respect to $\sigma \in \Sigma$ and $t \in [0,1)$.
\end{Lemma}

%%%%%%%%%%%%%%%%%%%%%%%%%%%%%%%%%%%%%
\begin{Lemma}\label{1lem5.3}
For every $\sigma \in \Sigma$ and $t \in [0, 1)$, the following holds
\[
\lim_{R \to \infty} \ \sup_{\sigma \in \Sigma,\, t \in [0,1)} 
\|g_t^\sigma\|_{0}^{2 \cdot 2^*_{\mu}} 
= \|u_t^\sigma\|_{0}^{2 \cdot 2^*_{\mu}}.
\]
\end{Lemma}

\begin{proof}
The claim is an immediate consequence of~\cite[Lemma~5.1(v)]{goel2020coron}.
\end{proof}\QED

  {In order to proceed further, we define the manifold $\mathcal{M}$ as follows
\begin{equation}\label{M}
\mathcal{M} =
\left\{
u \in \mathcal{X}^{1,2}_0(\Omega) :
\int_{\Omega}\int_{\Omega}
\frac{|u_+(x)|^{2^*_{\mu}} |u_+(y)|^{2^*_{\mu}}}{|x - y|^\mu}
\, dx\,dy = 1
\right\}.
\end{equation}
We define the Sobolev-type quotient
\begin{equation}\label{eqS}
S(u,\Omega) =
\frac{\displaystyle
\int_{\Omega}|\nabla u|^2\,dx
+
\int_{\mathbb{R}^{2n}}
\frac{|u(x)-u(y)|^2}{|x-y|^{n+2s}} \, dx\,dy}
{\left(
\displaystyle
\int_{\Omega}\int_{\Omega}
\frac{|u_+(x)|^{2^*_\mu}|u_+(y)|^{2^*_\mu}}{|x-y|^\mu}
\,dx\,dy
\right)^{1/2^*_\mu}},
\end{equation}
for all $u\in\mathcal X^{1,2}_0(\Omega)$ such that
\[
\int_{\Omega}\int_{\Omega}
\frac{|u_+(x)|^{2^*_\mu}|u_+(y)|^{2^*_\mu}}{|x-y|^\mu}\,dx\,dy>0 .
\] Moreover, set
\[
\tau(u) =
\left(
\int_{\Omega}\int_{\Omega}
\frac{|u_+(x)|^{2^*_{\mu}} |u_+(y)|^{2^*_{\mu}}}{|x - y|^\mu}
\, dx\,dy
\right)^{1/2^*_{\mu}}.
\]
\begin{Proposition}\label{p2.4}
For every $u\in\mathcal M$ there exists a neighborhood
$V_u\subset \mathcal X^{1,2}_0(\Omega)$ of $u$ such that
$v_+\not\equiv0$ for all $v\in V_u$.
Moreover, the Sobolev-type quotient $S(\cdot,\Omega)$
belongs to $C^1(V_u)$.
Furthermore, if $S'(u,\Omega)=0$, then there exists
$\lambda>0$ such that $I'(\lambda u)=0$.
\end{Proposition}
\begin{proof}
Let
\[
N(u)=
\int_{\Omega}|\nabla u|^2\,dx
+
\int_{\mathbb{R}^{2n}}
\frac{|u(x)-u(y)|^2}{|x-y|^{n+2s}}\,dx\,dy ,
\]
and
\[
F(u)=
\int_{\Omega}\int_{\Omega}
\frac{|u_+(x)|^{2^*_\mu}|u_+(y)|^{2^*_\mu}}{|x-y|^\mu}\,dx\,dy,
\qquad
D(u)=F(u)^{1/2^*_\mu}.
\]
It is straightforward to check that
$N\in C^1(\mathcal X^{1,2}_0(\Omega))$.
We now prove that $F\in C^1(\mathcal X^{1,2}_0(\Omega))$.
% Set
% \[
% g_u(x)=|u_+(x)|^{2^*_\mu}.
% \]
% Then
% \[
% F(u)=
% \int_\Omega\int_\Omega
% \frac{g_u(x)g_u(y)}{|x-y|^\mu}\,dx\,dy .
% \]
Since $2^*_\mu>1$, the map
$u\mapsto g=|u_+|^{2^*_\mu}$
is continuously Fr\'echet differentiable from
$\mathcal X^{1,2}_0(\Omega)$ into
$L^{\frac{2n}{2n-\mu}}(\Omega)$. 
% with derivative
% \[
% g_u'(x)[\varphi]
% =
% 2^*_\mu
% |u_+(x)|^{2^*_\mu-2}u_+(x)\varphi(x).
% \]
By the Hardy-Littlewood-Sobolev inequality,
% the bilinear form
% \[
% (g,g)\mapsto
% \int_\Omega\int_\Omega
% \frac{g(x)g(y)}{|x-y|^\mu}\,dx\,dy
% \]
% is continuous on
% \[
% L^{\frac{2n}{2n-\mu}}(\Omega)\times
% L^{\frac{2n}{2n-\mu}}(\Omega).
% \]
% Therefore $F$ is differentiable. 
% % and
% % \[
% % F'(u)[\varphi]
% % =
% % 2(2^*_\mu)
% % \int_\Omega
% % \left(
% % \int_\Omega
% % \frac{|u_+(y)|^{2^*_\mu}}{|x-y|^\mu}dy
% % \right)
% % |u_+(x)|^{2^*_\mu-2}
% % u_+(x)\varphi(x)\,dx .
% % \]
$F\in C^1(\mathcal X^{1,2}_0(\Omega))$.
Finally, since $F(u)=1$ for every $u\in\mathcal M$ and $F$ is continuous,
there exists $\delta>0$ such that $F(v)>\tfrac12$ whenever
$\eta(v-u)<\delta$.
Setting $V_u=B_\delta(u)$, we have $v_+\not\equiv0$ for all $v\in V_u$, and hence
$D(v)>0$, so that $S(\cdot,\Omega)\in C^1(V_u)$.
Since $S'(u,\Omega)(w)=0$, it follows that
\begin{align}
&\tau(u)\Bigg(
\int_{\Omega} \nabla u \cdot \nabla w\,dx
+
\int_{\mathbb{R}^{2n}}
\frac{(u(x)-u(y))(w(x)-w(y))}{|x-y|^{n+2s}}
\, dx\,dy
\Bigg) \notag \\
&\quad =
\eta(u)^2 \tau(u)^{1-2^*_{\mu}}
\int_{\Omega}\int_{\Omega}
\frac{|u_+(y)|^{2^*_{\mu}}
|u_+(x)|^{2^*_{\mu}-1} w(x)}
{|x-y|^\mu}
\, dx\,dy .
\end{align}
Thus,
\[
\int_{\Omega} \nabla u\cdot \nabla w\,dx
+
\int_{\mathbb{R}^{2n}}
\frac{(u(x)-u(y))(w(x)-w(y))}
{|x - y|^{n+2s}} \, dx\,dy
=
\frac{\eta(u)^2}{\tau(u)^{2^*_{\mu}}}
\int_{\Omega}\int_{\Omega}
\frac{|u_+(y)|^{2^*_{\mu}}
|u_+(x)|^{2^*_{\mu}-1} w(x)}
{|x-y|^\mu}
\, dx\,dy .
\]
Choosing $\lambda>0$ such that
\begin{equation}\label{lam}
\lambda^{2(2^*_{\mu}-1)}
=
\frac{\eta(u)^2}
{\displaystyle
\int_{\Omega}\int_{\Omega}
\frac{|u(x)|^{2^*_{\mu}} |u(y)|^{2^*_{\mu}}}{|x-y|^\mu}
\, dx\,dy},
\end{equation}
we obtain $I'(\lambda u)=0$.
This completes the proof.
\QED
\end{proof}}

   \begin{Lemma}\label{lem2.88}
Let $\{v_k\}_k \subset \mathcal{M}$ be a Palais–Smale sequence for the functional $S(\cdot, \Omega)$ at level $c$. Then, the sequence
\[
u_k := \lambda_k v_k, \quad \text{where } \lambda_k := \left(S(v_k, \Omega)\right)^{\frac{n-2}{2(n - \mu + 2)}},
\]
is a Palais–Smale sequence for the energy functional $I$ at the level 
\[
\frac{n - \mu + 2}{2(2n - \mu)} \, c^{\frac{2n - \mu}{n - \mu + 2}}.
\]
\end{Lemma}

\begin{proof} It is easy to see that for any $w\in\mathcal{X}^{1,2}_0(\Omega)$, we have
\begin{align}\label{eq:frechet-derivative}
\frac{1}{2} \left\langle S'(v_k, \Omega), w \right\rangle
&= \displaystyle\int_{\Omega} \nabla v_k\cdot \nabla w\,dx+\int_{\mathbb{R}^{2n}} \frac{(v_k(x) - v_k(y))(w(x) - w(y))}{|x - y|^{n + 2}}\,dx\,dy \notag \\
&\quad - \lambda_k^{2(2^*_{\mu}-1)} \int_{\Omega} \int_{ \Omega}
\frac{|(v_k)_+(y)|^{2^*_{\mu}} |(v_k)_+(x)|^{2^*_{\mu}-1}  w(x)}{|x - y|^\mu} \,dx\,dy.
\end{align}
Now, multiplying the above identity \eqref{eq:frechet-derivative} by $\lambda_k$, we get
\begin{equation}\label{eq:ps-transformed}
\begin{aligned}
\left\langle I'(u_k), w \right\rangle
&= \displaystyle\int_{\Omega} \nabla u_k\cdot \nabla w\,dx+\iint_{\mathbb{R}^{2n}} \frac{(u_k(x) - u_k(y))(w(x) - w(y))}{|x - y|^{n + 2s}}\,dx\,dy \\
&\quad - \int_{\Omega} \int_{ \Omega}
\frac{|(u_k)_+(y)|^{2^*_{\mu}} |(u_k)_+(x)|^{2^*_{\mu}-1}  w(x)}{|x - y|^\mu} \,dx\,dy.
\end{aligned}
\end{equation}
Since $v_k\in \mathcal{M}$, it follows 
\[
\lambda_k^{2(2^*_{\mu}-1)} = \eta(v_k)^2 = S(v_k, \Omega),
\]
that means $\lambda_k= \left(S(v_k, \Omega)\right)^{\frac{n-2}{2(n-\mu+2)}}.$
Now, since $v_k$ is a Palais–Smale sequence at level $c$, we have $S(v_k, \Omega) = c + o(1)$ as $k \to \infty$, which implies
\[
\lambda_k = (c + o(1))^{\frac{n - 2}{2(n - \mu + 2)}} \quad \text{which is bounded}.
\]
In particular, it follows that $\langle I'(\lambda_k v_k), w\rangle \to 0$ as $n\to\infty.$
Therefore, $u_k = \lambda_k v_k$ is bounded in $\mathcal{X}^{1,2}_0(\Omega)$.

Finally, to compute the energy level, we use the identity
\[
I(u_k) = \frac{1}{2} \eta(u_k)^2 - \frac{1}{2 \cdot 2^*_{\mu}} \int_{\Omega} \int_{ \Omega} \frac{|(u_k)_+(x)|^{2^*_{\mu}} |(u_k)_+(y)|^{2^*_{\mu}}}{|x - y|^\mu} \,dx\,dy.
\]
Then, we obtain
\[
\lim_{k \to \infty}I(u_k) = \left( \frac{1}{2} - \frac{1}{2 \cdot 2^*_{\mu}} \right) \lim_{k \to \infty}\lambda_k^{2 \cdot 2^*_{\mu}} = \frac{n - \mu + 2}{2(2n - \mu)} \, \lim_{k \to \infty}\lambda_k^{2 \cdot 2^*_{\mu}}.
\]
Substituting the expression of $\lambda_k$ from above, we finally get
\[
\lim_{k \to \infty} I(u_k) = \frac{n - \mu + 2}{2(2n - \mu)} \, c^{\frac{2n - \mu}{n - \mu + 2}}.
\]
This completes the proof.
\end{proof}
\QED

%%%%%%%%%%%%%%%%%%%%%%%%%%%%%%%%%%%%%%%%%%%%%%%%%%%%%%%%%%%%%%%%%%%

As a direct consequence of Proposition~\ref{p2.4}, Lemma \ref{lem2.88}, and the relationship between the functionals $S(\cdot, \Omega)$ and $I$, we obtain the following remark.

\begin{Remark}\label{cor:PS-for-S}
Since we proved that functional $I$ satisfies the Palais-Smale condition at all levels in the interval $(\beta, 2\beta)$ (see Lemma \ref{beta}),  it follows that  the functional $S(\cdot, \Omega)$ satisfies the Palais-Smale condition at all levels in the interval 
\[
\left(S_{H,L,C},\ 2^{\frac{n - \mu + 2}{2n - \mu}} S_{H,L,C} \right).
\]
\end{Remark}

%%%%%%%%%%%%%%%%%%%%%%%%%%%%

In order to prove Lemma \ref{lemma5.5}, we establish the following result, which is central to the study of the Coron-type problem discussed in this section.

\begin{Lemma}\label{lem4.6}
Let $u_0$ be given as in \eqref{eq5.2}. Then, for sufficiently large $n$, the inequality
\[
\frac{\|\nabla u_0\|_{L^2(\mathbb{R}^n)}^2 + (1-t)^{2-2s}[u_0]_s^2}
{\|u_0\|_{0}^2}
< 2^{\frac{n-\mu+2}{2n-\mu}}\, S_{H,L,C}
\]
holds, where
\(
\|u\|_0 
\) is defined by \eqref{no0}.
\end{Lemma}
  {
\begin{proof}
Since $u_0$ is an extremal function for $S_{H,L,C}$, we have
\[
\|u_0\|_0^2
=
\left(
\int_{\mathbb{R}^n}\int_{\mathbb{R}^n}
\frac{|(u_0)_+(x)|^{2^*_{\mu}} |(u_0)_+(y)|^{2^*_{\mu}}}{|x-y|^\mu}
\,dx\,dy
\right)^{\frac{1}{2^*_\mu}},
\]
and hence
\[
\frac{\|\nabla u_0\|_{L^2(\mathbb{R}^n)}^2}{\|u_0\|_0^2}=S_{H,L,C}.
\]
Therefore, it is enough to show that
\begin{equation}\label{eq4.4}
[u_0]_s^2
<
\Big(2^{\frac{n-\mu+2}{2n-\mu}}-1\Big)
\int_{\mathbb{R}^n} |\nabla u_0|^2\,dx .
\end{equation}
By the same argument as in \cite[Lemma 4.3]{chakraborty2025global}, we have the interpolation estimate
\[
[u_0]_s^2
\le
\frac{\omega_{n-1}2^{-s}}{s(1-s)}
\left(\int_{\mathbb{R}^n}|u_0|^2\,dx\right)^{1-s}
\left(\int_{\mathbb{R}^n}|\nabla u_0|^2\,dx\right)^s ,
\]
where
\[
\omega_{n-1}=\frac{2\pi^{\frac{n}{2}}}{\Gamma(\frac{n}{2})}
\]
is the surface measure of the unit sphere in $\mathbb{R}^n$.
Thus, \eqref{eq4.4} follows provided
\[
\frac{\omega_{n-1}2^{-s}}{s(1-s)}
\left(\int_{\mathbb{R}^n}|u_0|^2\,dx\right)^{1-s}
\left(\int_{\mathbb{R}^n}|\nabla u_0|^2\,dx\right)^s
\le
\Big(2^{\frac{n-\mu+2}{2n-\mu}}-1\Big)
\int_{\mathbb{R}^n}|\nabla u_0|^2\,dx .
\]
Equivalently,
\begin{equation}\label{eq0512}
\Bigg[
\frac{\omega_{n-1}2^{-s}}
{s(1-s)\big(2^{\frac{n-\mu+2}{2n-\mu}}-1\big)}
\Bigg]^{\frac{1}{1-s}}
\le
\frac{\displaystyle\int_{\mathbb{R}^n} |\nabla u_0|^2\,dx}
{\displaystyle\int_{\mathbb{R}^n} |u_0|^2\,dx }.
\end{equation}
Recall that $u_0$ is the bubble defined in \eqref{eq5.2}, so both the $L^2$ norm and the Dirichlet energy can be computed explicitly. 
% Moreover
% \(
% S_{H,L,C}\,C(n,\mu)^{\frac{n-2}{2n-\mu}}=S_n,
% \)
% with the same minimizers.
Now, it is more convenient to use the explicit formulas for the above integrals \eqref{eq0512}.
First,
\[
\int_{\mathbb{R}^n} |u_0(x)|^2\,dx
=
\omega_{n-1}
\int_0^\infty
\frac{r^{n-1}}{(1+r^2)^{n-2}}\,dr
=
\omega_{n-1}\,
\frac{\Gamma\!\left(\frac n2-2\right)\Gamma\!\left(\frac n2\right)}{2\Gamma(n-2)}
=
\pi^{\frac{n}{2}}\frac{\Gamma\!\left(\frac n2-2\right)}{\Gamma(n-2)}.
\]
Next, a direct computation gives
\[
\nabla u_0(x)=-(n-2)x(1+|x|^2)^{-n/2},
\]
and therefore
\[
|\nabla u_0(x)|^2=(n-2)^2|x|^2(1+|x|^2)^{-n}.
\]
Hence
\[
\int_{\mathbb{R}^n}|\nabla u_0(x)|^2\,dx
=
(n-2)^2\omega_{n-1}
\int_0^\infty \frac{r^{n+1}}{(1+r^2)^n}\,dr.
\]
Using the Beta-function identity,
\[
\int_0^\infty \frac{r^{n+1}}{(1+r^2)^n}\,dr
=
\frac12\,B\!\left(\frac n2+1,\frac n2-1\right)
=
\frac12\,
\frac{\Gamma\!\left(\frac n2+1\right)\Gamma\!\left(\frac n2-1\right)}{\Gamma(n)}.
\]
Substituting this into the previous expression and using
$\omega_{n-1}=\dfrac{2\pi^{n/2}}{\Gamma(n/2)}$, we obtain
\[
\int_{\mathbb{R}^n}|\nabla u_0(x)|^2\,dx
=
(n-2)^2\pi^{\frac n2}
\frac{\Gamma\!\left(\frac n2+1\right)\Gamma\!\left(\frac n2-1\right)}
{\Gamma(n/2)\Gamma(n)}.
\]
Using the identity $z\Gamma(z)=\Gamma(z+1)$, we have
\[
\Gamma\!\left(\frac n2+1\right)=\frac n2\,\Gamma\!\left(\frac n2\right),
\qquad
\Gamma\!\left(\frac n2-1\right)=\frac{\Gamma(n/2)}{\frac n2-1}.
\]
Hence
\[
\int_{\mathbb{R}^n}|\nabla u_0(x)|^2\,dx
=
n(n-2)\pi^{\frac n2}\frac{\Gamma\!\left(\frac n2\right)}{\Gamma(n)}.
\]
% Using the Beta-function identity,
% \[
% \int_0^\infty \frac{r^{n+1}}{(1+r^2)^n}\,dr
% =
% \frac12\,B\!\left(\frac n2+1,\frac n2-1\right)
% =
% \frac12\,
% \frac{\Gamma\!\left(\frac n2+1\right)\Gamma\!\left(\frac n2-1\right)}{\Gamma(n)},
% \]
% we obtain
% \[
% \int_{\mathbb{R}^n}|\nabla u_0(x)|^2\,dx
% =
% n(n-2)\pi^{\frac{n}{2}}\frac{\Gamma\!\left(\frac n2\right)}{\Gamma(n)}.
% \]
Consequently,
\[
\frac{\displaystyle\int_{\mathbb{R}^n} |\nabla u_0|^2\,dx}
{\displaystyle\int_{\mathbb{R}^n} |u_0|^2\,dx}
=
n(n-2)\,
\frac{\Gamma\!\left(\frac n2\right)\Gamma(n-2)}
{\Gamma\!\left(\frac n2-2\right)\Gamma(n)}.
\]
Therefore, a sufficient condition for \eqref{eq4.4} is
\begin{equation}\label{eq45}
\Bigg[
\frac{\omega_{n-1}2^{-s}}
{s(1-s)\big(2^{\frac{n-\mu+2}{2n-\mu}}-1\big)}
\Bigg]^{\frac{1}{1-s}}
\le
n(n-2)\,
\frac{\Gamma\!\left(\frac n2\right)\Gamma(n-2)}
{\Gamma\!\left(\frac n2-2\right)\Gamma(n)}.
\end{equation}
Now, using
\[
\frac{\Gamma\!\left(\frac n2\right)}{\Gamma\!\left(\frac n2-2\right)}
=
\left(\frac n2-1\right)\left(\frac n2-2\right)
\]
and
\[
\frac{\Gamma(n-2)}{\Gamma(n)}
=
\frac{1}{(n-1)(n-2)},
\]
the right-hand side of \eqref{eq45} simplifies to
\[
\frac{n(n-2)(n-4)}{4(n-1)}.
\]
Finally, by Stirling's formula,
\[
\Gamma(z+1)\sim \sqrt{2\pi z}\left(\frac{z}{e}\right)^z
\qquad (z\to\infty),
\]
we have
\[
\omega_{n-1}=\frac{2\pi^{n/2}}{\Gamma(n/2)}\to 0
\qquad\text{as }n\to\infty,
\]
while
\[
2^{\frac{n-\mu+2}{2n-\mu}}-1 \longrightarrow \sqrt{2}-1 >0.
\]
Hence, the left-hand side of \eqref{eq45} tends to $0$ as $n\to\infty$. On the other hand,
\[
\frac{n(n-2)(n-4)}{4(n-1)}\sim \frac n4 \to \infty
\qquad\text{as }n\to\infty.
\]
Therefore \eqref{eq45} holds for all sufficiently large $n$, which implies \eqref{eq4.4}. This completes the proof.
\end{proof}}
\QED
%%%%%%%%%%%%%%%%%%%%%%%%%%%%%%%%%%%%%%%%%%%%%%%%%%%%%%%%%%%%%%%%%%%

\begin{Lemma}\label{lemma5.5}
Let \( f_t^{\sigma}(x) := \dfrac{g_t^{\sigma}(x)}{\|g_t^{\sigma}\|_{0}} \) and \( f_0(x) := \dfrac{g_0(x)}{\|g_0\|_{0}} \). 
Then,
for $R > 1$ sufficiently large, there exist positive constants $S_1\in \mathbb{R}$, $n_0\in\mathbb{N}$ such that
\begin{equation}\label{S1}
\sup_{(\sigma, t)\in \Sigma \times [0,1)} S(f_t^\sigma, \Omega) < S_1 < 2^{\frac{n-\mu+2}{2n-\mu}} S_{H,L,C}, ~ \text{for all }~n\geq n_0.
    \end{equation}
\end{Lemma}
\begin{proof}
From Lemma \ref{Lem:seminormConvergence}, we know that 
\[
\|g_t^{\sigma} - u_t^\sigma\|_{\mathcal{D}^{1,2}(\mathbb{R}^n)} \to 0 
\quad \text{and} \quad 
[g_t^{\sigma} - u_t^\sigma]_s \to 0 \quad \text{as } R \to \infty.
\]
Using in addition, the relations
\[
\|u_t^\sigma\|_{\mathcal{D}^{1,2}(\mathbb{R}^n)} = \|u_0\|_{\mathcal{D}^{1,2}(\mathbb{R}^n)}
\quad \text{and} \quad 
[u_t^\sigma]_s = (1-t)^{1-s}[u_0]_s,
\]
we obtain
\[
S(f_t^{\sigma};\Omega) = S(g_t^{\sigma};\Omega) \longrightarrow S(u_t^\sigma;\mathbb{R}^n) 
= \frac{\|\nabla u_0\|_{L^2(\mathbb{R}^n)}^2 + (1-t)^{2-2s}[u_0]_s^2}
{\|u_0\|^2_{0}}
\]
as $R \to \infty$, uniformly in $\sigma \in \Sigma$ and for $0 \leq t < 1$.  Furthermore, by Lemma \ref{lem4.6}, there exists an integer $n_0 \equiv n_0(s) \in \mathbb{N}$ such that for all $n \geq n_0$ the strict inequality
\[
\frac{\|\nabla u_0\|_{L^2(\mathbb{R}^n)}^2 + (1-t)^{2-2s}[u_0]_s^2}
{\|u_0\|_{0}^2} 
< 2^{\frac{n-\mu+2}{2n-\mu}}\, S_{H,L,C}, 
\]
holds.  
This completes the proof. 
\end{proof}
\QED

%%%%%%%%%%%%%%%%%%%%%%%%%%%%%%%%%%%%%%%%%%%%%%%%%%%%%%%%%%%%%%%%%%%%%%%%%%%%%%%%%%%%%%%%%%9Proof of Theorem 1.1############
Having developed the necessary preliminary results, we are now ready to finalize the proof of Theorem \ref{mainthm}. To this end, we employ a deformation argument to lower the energy in a neighbourhood of the first level.
\paragraph{Proof of Theorem \ref{mainthm}.}
\noindent
Let us fix a value of $R$ and suppose, for contradiction, that problem \eqref{1} does not admit a positive solution. 
% By Lemma \ref{beta}, the functional $I$ satisfies the Palais--Smale condition $(PS)_{\Theta}$ on the space $\mathcal{X}^{1,2}_0(\Omega)$ for 
% \[
% \Theta \in \left( 
% \tfrac{1}{2}\left( \tfrac{n-\mu+2}{2n-\mu} \right) S_{H,L,C}^{\frac{2n-\mu}{\,n-\mu+2}},
% \; \left( \tfrac{n-\mu+2}{2n-\mu} \right) S_{H,L,C}^{\frac{2n-\mu}{\,n-\mu+2}}
% \right).
% \]
As we have (in the view of Remark \ref{cor:PS-for-S}), $S(\cdot, \Omega)$ satisfies the  Palais-Smale condition $(PS)_{\Theta}$ on the constraint manifold $\mathcal{M}$ (see~\eqref{M}) provided
\[
\Theta \in \left(S_{H,L,C},\, 2^{\frac{n - \mu + 2}{2n - \mu}} S_{H,L,C}\right).
\]
We will argue by contradiction. Suppose $S(\cdot, \Omega)$ does not admit a critical value within this range. 
By the Deformation Lemma \cite[Theorem~ II.311]{struwe2000variational}, for any $\Theta \in \left(S_{H,L,C},\, 2^{\frac{n - \mu + 2}{2n - \mu}} S_{H,L,C}\right)$, there exist $\rho > 0$ and flow  $\mathcal{G}_{\Theta}: \mathcal{M}\times [0,1] \to \mathcal{M}$ such that
\(
\mathcal{G}_{\Theta}\big( \mathcal{M}_{\Theta+\rho},1 \big) \subset \mathcal{M}_{\Theta-\rho},
\)
where $\mathcal{M}_\Theta = \{ u \in \mathcal{M} : S(u, \Omega) < \Theta \}$. 
Now, for an arbitrary $\varepsilon > 0$, by compactness, the interval $[S_{H,L,C} + \varepsilon,\, S_1]$  can be covered by finitely many intervals of the form $(\Theta_j - \rho_j,\, \Theta_j + \rho_j)$ and compose the corresponding deformations  $\mathcal{G}_{\Theta}$ to construct a flow \(\mathcal{G}: \mathcal{M}\times [0,1] \to \mathcal{M}\)  such that
\(
\mathcal{G}(\mathcal{M}_{S_1}, 1) \subset \mathcal{M}_{S_{H,L,C} + \varepsilon},
\)
where $S_1$ is defined in  \eqref{S1}.
Moreover, we may assume that during the deformation for all $t\in [0,1]$, we have
\[
\mathcal{G}(u, t) = u \quad \text{for all } u \in \mathcal{M} \text{ with } S(u, \Omega) < S_{H,L,C} + \tfrac{\varepsilon}{2}.
\]
\noindent Consider the \emph{center of mass} map $G : \mathcal{M} \to \mathbb{R}^n$ defined as
\begin{equation*}
G(u)= \frac{\int_{\Omega} x|\nabla u|^{2} {\rm d}x}{\int_{\om} |\nabla u|^2{\rm d}x}.
\end{equation*}
\noi
\noindent We claim that for any open neighborhood $U$ of $\bar{\Omega}$, there exists a $\varepsilon > 0$ such that
$$ G(\mathcal{M}_{S_{H,L,C} + \varepsilon}) \subset U.$$ \noi 
Suppose, on the contrary, that no such $\varepsilon > 0$ exists. Then, for every $k \in \mathbb{N}$, we have $G(\mathcal{M}_{S_{H,L,C} + \frac{1}{k}}) \not\subset U$. Consequently, there exists a sequence $\{ u_k \}$ with $u_k \in \mathcal{M}_{{S_{H,L,C}} + \frac{1}{k}}$ such that $G(u_k) \notin U$. \\
Since $S(u_k; \Omega) < {S_{H,L,C}}+ \frac{1}{k}$ and by definition $S(u_k; \Omega) \geq S_{H,L,C}$, we conclude
$$
\lim_{k \to \infty} S(u_k; \Omega) = S_{H,L,C},\quad
S(u_k;\om) \geq S_{H,L,C}.$$
This gives 
$$\displaystyle \lim_{k \to \infty} S(u_k;\om) = S_{H,L,C}$$
i.e., $\{u_k\}$ is a minimizing sequence for $S_{H,L,C}$.
 By Ekeland's Variational Principle \cite[Theorem 8.5]{willem2012minimax}, there exists a Palais-Smale sequence $\{v_k\}$ for $S(\cdot; \Omega)$ at the level $S_{H,L,C}$. It is easy to see that the sequence $\{v_k\}$ is bounded in $\X(\Omega)$, and up to a subsequence, $v_k \rightharpoonup v$ weakly in $\X(\Omega)$. From Lemma \ref{beta}, we know that the functional $I$ satisfies the Palais Smale condition for any $c \in (\beta, 2\beta)$. Since $\{v_k\}$ is a Palais-Smale sequence for $S(\cdot,\Omega)$ at level $S_{H,L,C}$, by Lemma \ref{lem2.88}, the rescaled sequence $u_k = \lambda_k v_k$ with $\lambda_k = (S(v_k,\Omega))^{\frac{n-2}{2(n-\mu+2)}}$ forms a Palais-Smale sequence for the energy functional $I$ at level $\beta = \frac{1}{2}\left(\frac{n-\mu+2}{2n-\mu}\right)S_{H,L,C}^{\frac{2n-\mu}{n-\mu+2}}$.
% Setting $w_k := 2 S_{H,L,C}^{\frac{n-2}{2(n-\mu+2)} } v_k$, we obtain a PS sequence for the functional $I$ at the level $\left( \tfrac{n-\mu+2}{2n-\mu} \right) S_{H,L,C}^{\frac{2n-\mu}{\,n-\mu+2}}$, see Lemma \ref{lem2.88}. Once again, Theorem \ref{PS_decomposition} ensures $I\left(2 S_{H,L,C}^{\frac{n-2}{2(n-\mu+2)} } v_k\right) \leq \left( \tfrac{n-\mu+2}{2n-\mu} \right) S_{H,L,C}^{\frac{2n-\mu}{\,n-\mu+2}}.$ 
% 
% By applying Theorem~8.4 of \cite{garain2022regularity}, and using the equivalence of critical points of $I$ and $S(\cdot, \Omega)$, we conclude that either $v \equiv 0$ or $v > 0$. Thus our assumption enforces $v \equiv 0$ in $\mathbb{R}^n$.
%  {maximum principle refernce chnage}
\smallskip
\noindent Moreover, since $\|v_k\|_{0} = 1$ and
$$
\int_{\Omega} |\nabla v_k|^2\,{\rm d}x + [v_k]_s^2 \to S_{H,L,C} \quad \text{as } k \to \infty,
$$
we infer that $\|\nabla v_k\|_{L^2(\Omega)}$ is bounded. By Lemma \ref{compact_embedding}, it follows that $[v_k]_s^2 \to 0$ as $k \to \infty$, and thus
$$
\|\nabla v_k\|^2_{L^2(\Omega)} \to S_{H,L,C}.
$$
\noindent Additionally, we consider the weak convergence in the sense of measures
$$
|\nabla v_k|^2 \, dx \rightharpoonup \mu, \quad 
\int_{\Omega} \left( \frac{|(v_k)_+(y)|^{2^*_{\mu}}}{|x - y|^\mu} \right) |(v_k)_+|^{2^*_{\mu}} \, dx \rightharpoonup \nu,
$$
as $k \to \infty$ in the space of Radon measures.
Then, using the identity $S^{-1}_{H,L,C} \|\mu\| = \|\nu\|^{2/2^*_{\mu}}$ and Lemma 1.40 in \cite{willem2012minimax}, we deduce that both $\mu$ and $\nu$ are concentrated at a single point, say $x^{(0)}\subset \bar{\Omega}$. Hence, up to a subsequence,
\begin{equation*}
\int_{\Omega} \left( \frac{|(v_k)_+(y)|^{2^*_{\mu}}}{|x - y|^\mu} \right) |(v_k)_+|^{2^*_{\mu}} \, dx\rightharpoonup \delta_{x^{(0)}}, \;\;
|\nabla v_k|^2 \, {\rm d}x \rightharpoonup S_{H,L,C}~ \delta_{x^{(0)}}.
\end{equation*}

\noindent It follows that
$$
G(v_k) = \frac{\int_{\Omega} x |\nabla v_k|^2 \, dx}{\int_{\Omega} |\nabla v_k|^2 \, dx} \to \frac{S_{H,L,C} x^{(0)}}{S_{H,L,C}} = x^{(0)}.
$$

\noindent Since the topology of $\mathbb{R}^{n}$ is normal, there exists an open set $V$ such that $\bar{\om} \subset V \subset \bar{V} \subset \subset U$. Because $G(u_k) \notin U$ and $|G(u_k) - G(v_k)| \to 0$ as $k \to \infty$, it follows,
$
x^{(0)} \notin \bar{\Omega},
$
yielding a contradiction. Thus, our claim is proved. Since $\mathbb{R}^{n}$ is a normal topological space, we can find an open set $V$ such that  
\[
\bar{\Omega} \subset V \subset \bar{V} \subset\subset U.
\]  
Because $G(u_k) \notin U$ while $|G(u_k)-G(v_k)| \to 0$ as $k \to \infty$, it follows that  
\(
x^{(0)} \notin \bar{\Omega},
\) 
which is a contradiction. Therefore, the claim holds.

 Owing to the smoothness of $\Omega$, there exists an open neighborhood $U_0$ of $\bar{\Omega}$ such that every point $p \in U_0$ admits a unique nearest projection $q=\pi(p)\in \bar{\Omega}$, where the projection map $\pi: U_0 \to \bar{\Omega}$ is continuous.
Let $\varepsilon>0$ be chosen as above for this neighbourhood $U_0$, and recall the deformation flow $\mathcal{G}:\mathcal{M}\times [0,1]\to \mathcal{M}$ constructed earlier. We now define a map $h:\Sigma \times [0,1]\to \bar{\Omega}$ by  
\[
h(\sigma,t):=\pi\!\left(G\!\left(\mathcal{G}(f_t^\sigma,1)\right)\right),
\]  
where $f_t^\sigma$ is the two-parameter family introduced previously.

The map $h$ is well defined and continuous, and by the Lemma \ref{lemma5.5} it satisfies  
\[
\begin{aligned}
&h(\sigma,0)=\pi\!\left(G\!\left(\mathcal{G}(f_0,1)\right)\right)=:x^{(0)}\in \bar{\Omega}, \quad \forall \sigma\in \Sigma,\\
&h(\sigma,1)=\sigma, \quad \forall~ \sigma~\in \Sigma.
\end{aligned}
\]  
Consequently, the map $h$ defines a contraction of the unit sphere $\Sigma$ onto the point $x^{(0)} \in \overline{\Omega}$, which contradicts condition~\eqref{eq1.7}. Hence, our assumption is false, and it follows that $S_{H,L,C}(\cdot, \Omega)$ possesses a critical value. This means that there exists a function $u \in \mathcal{X}^{1,2}_0(\Omega)$ such that $u$ is a nonnegative solution of~\eqref{2}. 
Therefore, the proof of Theorem \ref{mainthm} is complete.
\QED

% \begin{thebibliography}{99}
% \bibitem{mukherjee2017fractional} T. Mukherjee, K. Sreenadh, \textit{Fractional Choquard equation with critical nonlinearities}, Nonlinear Differ. Equ. Appl. 24 (2017), 24-63.

% \bibitem{su2022regularity} X. Su, E. Valdinoci, Y. Wei, J. Zhang, \textit{Regularity results for solutions of mixed local and nonlocal elliptic equations}, Math. Z. 302 (2022), 1855-1878.

% \bibitem{biagi2021semilinear} S. Biagi, S. Dipierro, E. Valdinoci, E. Vecchi, \textit{Semilinear elliptic equations involving mixed local and nonlocal operators}, Proc. Roy. Soc. Edinburgh Sect. A 151 (5) (2021), 1611-1641.

% \bibitem{ros2014extremal} X. Ros-Oton, J. Serra, \textit{The extremal solution for the fractional laplacian}, Cal. Var., 50 (2014), 723-750.

% \bibitem{stein1970singular} E. Stein, \textit{Singular Integrals and Differentiability Properties of Functions}, Princeton Mathematical Series, vol. 30. Princeton University Press, Princeton (1970).
% \end{thebibliography}

\vspace{0.5cm}

\noindent{\bf Acknowledgments:} 
%Jacques Giacomoni was partially funded by IFCAM (Indo-French Centre for Applied Mathematics) IRL CNRS 3494.	
Tuhina Mukherjee acknowledges the financial support provided by CSIR-HRDG with sanction No. 25/0324/23/EMR-II. Lovelesh Sharma acknowledges the financial support provided by the Raman Charpak Fellowship, jointly funded by the Government of India and the Government of France.

\bibliographystyle{abbrv}
\bibliography{references}

 \end{document}